\documentclass [12pt,oneside] {article}
\usepackage{latexsym,amsfonts,amssymb,amscd,amsmath,amsthm,mathrsfs,stmaryrd} 
\usepackage{graphicx}
\usepackage [all]{xy}
\SelectTips{cm}{12}

\theoremstyle{plain}
\newtheorem {Thm} {Theorem}[section]
\newtheorem {Lem}[Thm] {Lemma}
\newtheorem {Prop}[Thm] {Proposition}
\newtheorem {Cor}[Thm] {Corollary}

\theoremstyle {definition}
\newtheorem {Def}[Thm] {Definition}

\newtheorem {Rem}[Thm] {Remark}
\newtheorem {Exa}[Thm] {Example}
\newenvironment{Pf}[1]{{\noindent\sc Proof #1:}}{\qed\\}
\newcommand {\specialmap} [4] {{#1\negmedspace : #2 #3 #4}}

\newcommand {\map} [3] {\specialmap {#1} {#2}{\to} {#3}}
\newcommand {\longmap} [3] {\specialmap {#1} {#2}{\longrightarrow} {#3}}

\newcommand {\simmap} [3] {\specialmap {#1} {#2}{\stackrel {\sim{\phantom{.}}} 
            {\longrightarrow}} {#3}}

\newcommand {\id} {\operatorname{id}}

\newcommand {\at}[1] {\arrowvert_{#1}}

\newcommand {\ind} {\operatorname{ind}}
\renewcommand {\(} {\left(}
\renewcommand {\)} {\right)} 

\renewcommand {\geq} {\geqslant}

\newcommand {\sub} {\subseteq}

\newcommand {\CC} {\mathbb C}
\newcommand {\on}[1] {\operatorname{#1}}

\newcommand {\Gpdf} {{\mathcal G\!\on{pd}_{\on{fin}}}}

\newcommand {\GSetsf} {{G\mathcal S\!\on{ets}_{\on{fin}}}}

\newcommand {\hinf} {{H_\infty}}

\renewcommand {\ind}[2] {\on{ind}_{#1}^{#2}}

\renewcommand {\leq} {\leqslant}

\newcommand {\Lt} {{\langle\tau,1\rangle}}

\newcommand {\mA} {{\mathcal{A}}}
\newcommand {\mF} {{\mathcal{F}}}

\newcommand {\pt} {\on{pt}}

\newcommand {\QQ} {\mathbb{Q}}
\newcommand {\res}[2] {\on{res}^{#1}_{#2}}

\newcommand {\Sj} {{S_j}}

\newcommand {\Sn} {{S_n}}

\renewcommand {\SS} {\mathbb{S}}

\newcommand {\Stab} {\on{Stab}}

\newcommand {\tensor}{\otimes}
 
\newcommand {\Tate} {{\on{Tate}}}

\newcommand {\ZZ} {\mathbb Z}

\newcommand {\lps} {[\! [}
\newcommand {\rps} {]\! ]}

\newcommand {\ps}[1] {\lps #1\rps}
\newcommand {\inv}{^{-1}}

\newcommand {\mC} {\mathcal C}
\newcommand {\mI} {\mathcal I}
\newcommand {\mL} {\mathcal L}
\newcommand {\mM} {\mathcal M}

\newcommand {\ol}[1]{\overline{#1}}
\newcommand {\ul}[1]{\underline{#1}}

\renewcommand {\ul}[1] {\underline{#1}}
\newcommand{\uF} {{\ul F}}
\newcommand{\uG} {{\ul G}}
\newcommand{\uH} {{\ul H}}
\newcommand{\uK} {{\ul K}}
\newcommand{\uP} {{\ul P}}
\renewcommand {\tilde}[1] {\widetilde{#1}}
\newcommand{\lwr}{\smallint}
\newcommand{\smalllwr}{\smallint}

\title{Global Mackey functors with operations and $n$-special lambda rings}
\author{Nora Ganter\thanks{This project was initiated during a SQuaRE workshop at the
  American Institute of Mathematics in the summer of 2008. At that time, the author
  was affiliated with Colby College and supported by NSF-grant
  DMS-0504539. Presently, the author is supperted by an Australian
  Research Fellowship and ARC Discovery Grant DP1095815.}  
  \\ The University of Melbourne} 
\date {\today}
\begin{document}
\maketitle
\tableofcontents
\section{Introduction}
Mackey functors were introduced by Dress and Green
\cite{Dress:Contributions}, \cite{Dress:Notes}, \cite{Green:Axiomatic}, following earlier
ideas of Bredon and Lam 
\cite{Bredon:Equivariant_cohomology:LNM}, \cite{Lam:Induction_theorems}. 
Excellent contemporary expositions are \cite{Webb:A_guide} and
\cite{Panchadcharam:Categories_of_Mackey_functors}.  
Inspired by the very first steps of the theory, we take the
point of view in \cite{Bredon:Equivariant_cohomology:LNM},
that these theories formalize the algebraic properties
of coefficient systems of equivariant (generalized) cohomology
theories. 

Many prominent examples of Mackey functors arise in this manner
or in a closely related context. For instance, 
the representation rings form the coefficients of topological
$K$-theory, group cohomology 
forms the coefficients of ordinary cohomology, and Burnside rings are
closely related to cohomotopy. 
We limit our attention to finite groups. 

Over the last decades,
our understanding of generalized cohomology has
evolved. Cohomology operations have emerged as an important part
of the theory. 
There is, in particular, Atiyah's work, putting the Adams operations
in $K$-theory into a common framework with
Steenrod's operations in cohomology \cite{Atiyah:Power_operations},
\cite{Adams:Vector_fields}, \cite{Steenrod:Cohomology_operations}. 
Axioms for these {\em `power operations'} were formulated by Bruner,
May, McClure and Steinberger\footnote{Bruner, May, McClure
  and Steinberger work in the Borel-equivariant
context, but the formulation for other equivariant theories is an easy
translation exercise, see for instance \cite{Ganter:Orbifold_genera}
or \cite{Ganter:Orbifold_K-Tate}.} in
\cite{Bruner:May:McClure:Steinberger},
and the modern foundations of homotopy theory
are set up in such a way, that `nice enough' multiplicative
cohomology theories automatically have power operations.
We would especially like to draw the reader's attention to the recent work of
Stefan Schwede \cite{Schwede:Orthogonal_spectra}. We understand that Schwede's {\em
  or\-tho\-go\-nal spectra} define genuinely equivariant theories with power operations,
whose coefficient systems form a large class of
examples for the formalism formulated below. 

A {\em global power functor} is a global Mackey functor with the extra
structure of power operations. This definition captures the
algebraic properties of the coefficient system of a (genuinely)
equivariant theory with cohomology operations. 
We note that there is
recent related work by Strickland on {\em Tambara functors}
\cite{Strickland:Tambara_functors} and that Schwede has independently
arrived at a notion of global power functor very similar to ours. 

The paper is organized as follows:
systematically using the language of groupoids, we start with an
exposition of the theory of {\em global Mackey and Green 
  functors}. These are discussed in
\cite{Webb:A_guide}, where Webb attributes the definition to Bouc
\cite{Bouc:Foncteurs}, referencing also older work of Symonds
\cite{Symonds:A_splitting}. What is different in our setup is the groupoid
point of view. This is by no account a new idea, see for instance
\cite{Strickland:K(N)-local} and especially 
\cite{Panchadcharam:Categories_of_Mackey_functors}. Since we do,
however, heavily use this formulation, we felt 
it worth to spell out the axioms and to give a full comparison to the definition of
global Mackey functor found in \cite{Webb:A_guide}. 

We then proceed to define global power functors. 
Given such a global power functor, one can study rings with similar
operations. For instance, special $\lambda$-rings are rings with 
operations similar to those on the representation rings $R(G)$. 
The precise concept is that of a $\tau$-ring, formulated by
Hoffman, who proved that a $\tau$ ring with respect to $R$ is the same
thing as a special $\lambda$-ring \cite{Hoffman:tau-rings}. 
Retelling Hoffman's story with $R$ replaced by a different global
power functor $M$, we arrive at a theory of rings with operations
parametrized by elements of 
$\bigoplus\limits^\wedge_{n\geq 0}M(S_n)$ or, depending on some
technical properties of $M$, by elements of
$\bigoplus\limits^\wedge_{n\geq 0}M(S_n)^*$.
These {\em $\tau$-rings with respect to $M$} turn out to be
non-special $\lambda$-rings with additional structure. 

A key player in \cite{Hoffman:tau-rings} is the
Schur-Weyl isomorphism
\begin{eqnarray}\label{eq:Schur-Weyl-isomorphism}
  \bigoplus_{n\geq 0}^\wedge R(S_n) &
  \xrightarrow{\phantom{xx}\cong\phantom{xx}} & \lim_{m\in \mathbb N}\,R(U(m)),
\end{eqnarray}
defined and studied via $\tau$-operations, and we will discuss some attempts
to generalize \eqref{eq:Schur-Weyl-isomorphism}.

The example of the Burnside rings has been studied extensively, the
resulting `$\tau$-rings' have become known as $\beta$-rings, see in particular
\cite{Rymer:Power_operations}, but also
\cite{Boorman:S-operations},
\cite{Ochoa:Outer_plethysm}, \cite{Morris:Wensley:Adams_operations},
\cite{Vallejo:The_free_beta-ring}, \cite{Guillot:Adams_operations}. 
Other examples include $n$-special $\lambda$-rings, involving
class functions on $n$-tuples of commuting elements, `subgroup class
functions', where permutation representations have their characters, and an elliptic
picture, where we suggest the construction of an elliptic analogue of
the Schur-Weyl map \eqref{eq:Schur-Weyl-isomorphism}, taking values in
Looijenga's ring of symmetric theta functions \cite{Looijenga:Root_systems}. 
Rezk has independently developed a
theory of {\em `elliptic $\lambda$-rings'}, and we make contact with his work in \ref{sec:Rezk}. 

A variation where
$R(S_n)$ is replaced with $R^-(S_n)$, the Grothendieck group of
projective representations with Schur cocycle, gives rise to the
notion of super $\lambda$-ring, fitting Sergeev-Yamaguchi duality into
Hoffman's framework for \eqref{eq:Schur-Weyl-isomorphism}.


\subsubsection*{Acknowledgments}
Many many thanks go to Mikhail Kapranov for suggesting this project. A
first draft of this paper, dating back to 2008, was written in
collaboration with Mikhail, 
and many of the ideas presented here are his. A warm thanks also goes to
Charles Rezk for generously sharing his ideas and his unpublished
work, and likewise to Jack Morava and Rekha Santhanam.
I like to thank Arun Ram for helpful conversations and, in particular for bringing my
attention to the reference \cite{Jozefiak:Relating_spin}. I would like
to thank Stefan Schwede and Eugene Lerman for
helpful conversations. Finally I would like to thank Elena Ganter, who
helped with some of the typing. 
%
\section{Globally defined Mackey functors}\label{global-Mackey-Sec}
\subsection{Finite groupoids} 
For a detailed introduction to groupoids see 
\cite{Moerdijk:Orbifolds_as_groupoids} or \cite{Strickland:K(N)-local}.
A finite groupoid is a
category $\uG$ with finitely many objects and finitely many morphisms and 
such that each morphism is invertible. 
We will write 
$G_0$ for the set of objects of $\uG$ and $G_1$ for the set of
morphisms. To avoid confusion with maps between groupoids, we
will refer to elements of $G_1$ as {\em arrows}. 
Let $g\in G_1$ be an arrow. We will
write $s(g)$ for the source and $t(g)$ for the target of $g$.
Given an object $x\in
G_0$, we will refer to the automorphism group of $x$ as the 
{\em stabilizer} $\on{Stab}(x)$ of $x$.
We write $[\uG]$ for the set of isomorphism classes of $\uG$.
If $G$ is a finite group, we will write $G$ also for the groupoid with
$G_0=pt$ and $G_1 = G$. 
We choose this over the more common notation $\pt/\!\!/ G$, 
since Mackey functors are traditionally applied to groups. So, for instance, the
representation ring of $G$ remains $R(G)$, rather than $R(\pt/\!\!/
G)$.

A {\em map of groupoids} $\map f{\ul H}{\uG}$ is a functor. 
We will write $\Gpdf$ for the category of finite groupoids and maps of
groupoids. 
One can (and should) view $\Gpdf$ as 
$2$-category with the natural isomorphisms as 2-morphisms, but
in this paper we will not emphasize the 2-categorical point of view.
A map $f$ is
an {\em equivalence} if it is an equivalence of categories, and $f$ is
called {\em faithful} if it is injective on stabilizers.

Fix a finite group $G$. Let $\GSetsf$ be the category of
(left) $G$-sets. We write 
\begin{eqnarray*}
{\GSetsf}&\longrightarrow &{\Gpdf}\\
{X}&\longmapsto&{G\ltimes (-)}  
\end{eqnarray*}
for the functor that sends a $G$-set $X$ to its translation
groupoid. To fix notation, we say that $G\ltimes X$ has objects $X$
and an arrow
$$gx\leftarrow x$$ 
for each pair $(g,x)$ in $G\times X$.
\begin{Exa}
The translation groupoid of the one point space is isomorphic to the
group $G$ itself. More generally, let $H$ be a subgroup of $G$, and
consider the left G-set $G/H$. The translation groupoid of $G/H$ is
canonically equivalent to the group $H$.  
\end{Exa}

\begin{Def}[Inertia groupoids]
  Let $\uG$ be a groupoid. The {\em inertia groupoid} 
  $$\mI(\uG) = \hom(\ZZ,\uG)$$ (sometimes also
  called {\em loop groupoid})  
  of $\uG$ has as objects the automorphisms of $\uG$. An arrow in $\mI(\uG)$
  from $g\in\Stab(x)$ to $g'\in\Stab(y)$ is an arrow $h\in G_1$ from $x$ to
  $y$ satisfying $g' = hg h\inv$. Composition is defined in the
  obvious way.
\end{Def}
\begin{Def}[Strict pullbacks] Consider a diagram of groupoids
  \begin{equation}
    \label{diagram-Eqn}
    \uK\stackrel\delta\longrightarrow \uG\stackrel\gamma\longleftarrow\uH.    
  \end{equation}
  The pullback groupoid $\uP$ of \eqref{diagram-Eqn} has object set
  $$P_0= K_0\times_{G_0}H_0,$$
  and similarly for arrows and the source, target and composition maps.
\end{Def}
  We have a commuting diagram
  $$
    \xymatrix{\uP\ar[r]^\alpha\ar[d]_\beta&\uH\ar[d]^\gamma\\
    \uK\ar[r]_\delta&\uG,
    }
  $$ 
  and $\uP$ is universal with respect to the
  property that it fits into this diagram.
\begin{Def}
  A diagram that is isomorphic to one obtained
  in this fashion is called a {\em pullback diagram}.  
\end{Def}
\begin{Def}[Fibered products]
  The {\em homotopy pullback} or {\em fibered product} $\uF$ of
  \eqref{diagram-Eqn} 
  has objects 
  $$
    F_0 := \{(y,g,z)\mid y\in K_0, \medskip z\in H_0,\medskip \map
    g{\delta(y)}{\gamma(z)}\} 
  $$
  and arrows 
  $$
    \xymatrix{y\ar[d]_k&
      \delta(y)\ar[r]^g\ar[d]_{\delta(k)}& \gamma(z)\ar[d]^{\gamma(h)} 
      &z\ar[d]^h \\ 
     y'&
      \delta(y')\ar[r]_{g'}& \gamma(z') &
      z' .}
  $$
  Composition of arrows is defined in the obvious way.
\end{Def}
The projections $\map\alpha\uF\uH$ and $\map\beta\uF\uK$ make the
{\em homotopy pullback diagram}
  $$
    \xymatrix{\uF\ar[r]^\alpha\ar[d]_\beta&\uH\ar[d]^\gamma\\
    \uK\ar[r]_\delta&\uG,
    }
  $$ 
commute up to a natural transformation
$$
  \gamma\circ\alpha \stackrel{\cong\phantom{..}}\Longrightarrow \delta\circ\beta,
$$
and the triple $(\uF,\alpha,\beta)$ is universal with respect to that
property.
\begin{Def}
  We will refer to any diagram that is equivalent to a diagram obtained
  in this fashion as a {\em homotopy pullback diagram}.  
\end{Def}
Note that the pullback $\uP$ is contained in $\uF$ as the full subgroupoid whose
objects are of the form $(y,1,z)$.
Fibered products are well-behaved under
equivalences, whereas pullbacks are not. 
\begin{Exa}[{See
    \cite[Chap.0]{Panchadcharam:Categories_of_Mackey_functors}}]\label{double-coset-Exa}  
  Assume that $K$, $G$, and $H$ are finite groups and
  that $\delta$ and $\gamma$ are injective maps of groups. 
  Then the isomorphism classes of the fibered product $\uF = K\times_G
  H$ are in one to one 
  correspondence with the double cosets $H\backslash G/K$. For $g\in
  G$ the stabilizer of $g$ in $\uF$ is isomorphic to $$H\cap (gKg\inv).$$
  Under this identification the map $\map\alpha\uF
  H$ becomes the inclusion of $H\cap (gKg\inv)$ in $H$. Similarly,
  $\beta$ becomes the inclusion of $(g\inv H g)\cap K$ in $K$, and the
  isomorphism that connects the source of $\alpha$ with the source of
  $\beta$ is conjugation by $g$. 
\end{Exa}

\begin{Rem} 
  The functor that sends
  \eqref{diagram-Eqn} to its pullback square is the
  right Kan extension along the map of underlying diagrams. 
  With an appropriate model structure on $\Gpdf$
  \cite[6.12]{Strickland:K(N)-local},
  the homotopy pullback becomes the right homotopy Kan
  extension along the same map of diagrams. In other words, forming the
  homotopy pullback (square) is the right derived functor of forming the
  pullback (square).
%
\end{Rem}

\begin{Prop}\label{G-set-pullback-Prop}
  The functor $G\ltimes (-)$ sends pullback diagrams in $\GSetsf$ to
  homotopy pullback diagrams in $\Gpdf$.
\end{Prop}
\begin{Pf}{}
In fact, the functor $G\ltimes(-)$ preserves pullback squares.
We claim that in this case the inclusion map $\map\iota{\uP}\uF$ is an
equivalence of groupoids. 
  We need to show that $\iota$ is essentially surjective. Pick an
  object $(y,g,z)$ of $\uF$. Then the arrow
$$
  \xymatrix{gy\ar[d]_{g\inv}&
    \delta(gy)\ar[r]^{1}\ar[d]_{g\inv}& \gamma(z)\ar[d]^{1} 
    &z\ar[d]^{1} \\ 
   y&
    \delta(y)\ar[r]_{g}& \gamma(z) &
    z }
$$
gives the desired isomorphism from an object in the image of $\iota$.
\end{Pf}
\subsection{Mackey functors}
  Let $\mA$ be an additive category. 
  Let $M$ be a pair
  $M=(M_*,M^*)$ of functors 
  $$
    \Gpdf\longrightarrow \mA,
  $$
  agreeing on objects, where $M_*$ is covariant and $M^*$ is contravariant.
  We write 
  $$
    M(\uG):=M_*(\uG) = M^*(\uG)  
  $$
  for the effect of $M$ on objects and 
  $$\phi_*:= M_*(\phi)\quad \text{  and  }\quad \phi^* := M^*(\phi)$$
  for its effect on morphisms.
  We call $M$ a {\em globally defined Mackey functor} if
  the following axioms are satisfied:
  \begin{description}
  \item[Coproducts:] $M_*$ preserves coproducts, and $M(\emptyset) = 0$.

\bigskip
  \item[Natural isomorphisms:] If $\phi$ and $\psi$ are two naturally
    isomorphic maps of
    groupoids $\uG\to\uH$, then we have $\phi_*=\psi_*$ and
    $\phi^*=\psi^*$.

\bigskip
   \item[Equivalences:] If $\simmap\phi\uG\uH$ is an equivalence of
     groupoids, then $\phi_*$ and $\phi^*$ are inverse
     isomorphisms.

\bigskip
   \item[Fibered Products:] For any diagram of groupoids
     $$
       \xymatrix{\uF\ar[r]^\alpha\ar[d]_\beta&\uH\ar[d]^\gamma\\
       \uK\ar[r]_\delta&\uG,
       }
     $$ 
     where $\uF$ is the {\em fibered product}
     of $\uH$ and $\uK$ over $\uG$, we have the {\em push-pull identity} 
     $$
       \delta^*\circ \gamma_* = \beta_*\circ \alpha^*.
     $$ 
  \end{description} 

A {\em natural transformation} $\map\eta MN$ of Mackey functors is a family of maps
\begin{eqnarray*}
  \eta_{\uG}\negmedspace : M(\uG) & \longrightarrow & N(\uG)
  \quad\quad \uG\in ob(Gpd)
\end{eqnarray*}
that is natural with respect to both, the co- and the contravariant
structure of $M$ and $N$.

\medskip
The following lemma is an immediate consequence of 
Proposition \ref{G-set-pullback-Prop}. 
\begin{Lem}
  Let $M$ be a globally defined Mackey functor with values in
  $\mA$. Then the composite
  $$M\circ G\ltimes(-)
  :\GSetsf\longrightarrow \mA$$
  is a Mackey functor for $G$ in the sense of Dress.
\end{Lem}

In the appendix, we prove that
our definition becomes equivalent to 
Webb's definition of globally defined Mackey functor
\cite[8]{Webb:A_guide} if one adds one more axiom to the above list: 
\begin{description} 
\item[Surjections axiom:] Let $\map\phi GH$ be a surjective map of
  groups. Then
  $$\phi_*\phi^* = \id_{M(H)}.$$ 
\end{description} 

\subsection{Consequences of the axioms}
\begin{enumerate}
\item Let $G$ be a finite group and let $H$ be a subgroup of $G$.
Then the canonical equivalence of groupoids
$$
  H \stackrel{\sim}{\longrightarrow} {\underline{G/H}},
$$
induces an isomorphism
$$
  M(H)\stackrel\cong\longrightarrow M(\ul{G/H}).
$$

\bigskip
\item Let $G$ be a finite group and let $\map{c^g}GG$ be an inner
  automorphism of $G$. Then $c^g$ is naturally isomorphic to
  $\id_G$, and hence $$(c^g)_* = \id_{M(G)} = (c^g)^*.$$

\bigskip 
%
\item 
Let $\simmap\alpha\uG\uH$ be an equivalence of groupoids with
quasi-inverse

\noindent
$\simmap\beta\uH\uG$. Then the composite $\alpha\circ\beta$ is
naturally isomorphic to $\id_\uH$, and we get
$$
  \alpha_* = \(\beta_*\)\inv = \beta^*.
$$

\bigskip
\item Let $\uG = \uH\sqcup\uK$, and let $\map{\ul\alpha}\uH\uG$ and
  $\map{\ul\beta}\uK\uG$ be the inclusions. 
  Then $\alpha^*$ and $\beta^*$ are the projections onto the
  summands of 
  $$M(\uG)\cong M(\uH)\oplus M(\uK),$$  
  while $\alpha_*$ and $\beta_*$ are the inclusions of the summands.
\end{enumerate}

\medskip
\begin{Pf}{of (4)}
  The fibred product
  $\uK\times_\uG\uH$ is empty, so we get
  $$\beta^*\alpha_* = 0 \quad\text{and}\quad\alpha^*\beta_* = 0.$$
  Further, we have $\uH\times_\uG\uH \cong \uH$ and similarly for
  $\uK$, implying $$\alpha^*\alpha_* =
  \id_{M(\uH)} \quad\text{and}\quad \beta^*\beta_* = \id_{M(\uK)}.$$    
\end{Pf}{}
%
\subsubsection{Variations}
In \cite{Webb:A_guide}, Webb describes
various variations of the definition of global Mackey functor. His
definition depends on two classes of finite groups, which he 
denotes $\mathcal X$ and $\mathcal Y$.
Above we
have only formulated the strongest case, where both classes contain all
finite groups.
One can adapt our definition to other cases
by imposing the appropriate conditions on stabilizers. 
For instance, for the case
where $\mathcal Y$ only contains the trivial
group, we can weaken our definition by requiring
 $M_*(\alpha)$ to be defined only for faithful $\alpha$.
\subsection{Examples}
\subsubsection{Burnside rings} 
  Let $Sets^{f}$ denote the category of finite sets. Then
  the category of {\em $\uG$-sets} (or {\em `sheaves in finite sets'} over $\uG$) is the
  functor category
  \begin{eqnarray*}
    Sh(\uG) & := & \mF\! un(\uG,Sets^{f}).
  \end{eqnarray*}

  Let $\map\phi\uG\uH$ be a map of finite groupoids.
  Then we have an adjoint functor pair
  \begin{eqnarray*}
    \phi_* : Sh(\uG) & \rightleftharpoons & Sh(\uH)
    : \phi\inv,
  \end{eqnarray*}
  with 
  \begin{eqnarray*}
    \phi\inv\mF = \mF\circ \phi & \text{and} & \phi_*\mF = RKan_\phi\mF
  \end{eqnarray*}
  (left Kan-extension, compare \cite[6.3]{Moerdijk:Orbifolds_as_groupoids}).
  The Burnside ring of $\uG$ is the Grothendieck group
  \begin{eqnarray*}
    A(\uG) & = & \(Sh(\uG)/\cong\medspace ,\medspace\sqcup\)^{gp}.
  \end{eqnarray*}
  The discussion in \cite{Panchadcharam:Categories_of_Mackey_functors} implies that
  $A$ is a global Mackey functor satisfying the surjectivity axiom. 
  \begin{Exa}
    If $G$ is a finite group, then $Sh(G)$ is naturally identified
    with the category of finite left-$G$-sets,
    and $A(G)$ is the Burnside ring of $G$.
  \end{Exa}
%
%
\subsubsection{Group cohomology} The Borel construction functor (c.f.\
\cite{Moerdijk:Orbifolds_as_groupoids}) 
\begin{eqnarray*}
   {\on{Borel}}\negmedspace:{\Gpdf}&\longrightarrow &{\mathcal{S}}
\end{eqnarray*}
  from the category of (finite) groupoids into the stable homotopy
  category has transfers along faithful maps of groupoids, making
  \begin{eqnarray*}
    \uG&\longmapsto& H^*(B\uG;\ZZ)
  \end{eqnarray*}
  a global Mackey functor, where the $\phi_*$ are only defined for
  faithful $\phi$. 
  The degree zero part 
  \begin{eqnarray*}
    H^0(B\uG;\ZZ) &\cong & \ZZ^{[\uG]}
  \end{eqnarray*}
  counts the number of isomorphism classes of $\uG$. 
  \begin{Lem}
  Let $f$ be a
  $\ZZ$-valued function on $[\uG]$, and let $\map\phi\uG\uH$ be a
  faithful map. Then $$\longmap{\phi_*f}{[\uH]}\ZZ$$ satisfies
  \begin{eqnarray}
    \label{eq:H^0-transfer}
    \frac{(\phi_*f)(y)}{|aut_{\uH}(y)|} & = &    
    \sum_{[x]\mapsto [y]} \frac{f(x)}{|aut_{\uG}(x)|}.
  \end{eqnarray}
  \end{Lem}
  \begin{proof}
    This follows from the homotopy pullback square
    $$
      \xymatrix{
      {\coprod\limits_{[x]\mapsto[y]}aut(y)/aut(x)}\ar[rr]\ar[d]&&\uG\ar[d]\\
      \pt\ar[rr]^y&&\uH.
      }
    $$
  \end{proof}
  We may replace $\ZZ$ with any other abelian ring $R$. If $R$
  is a $\QQ$-algebra then $H^0(-;R)$ possesses transfers along all maps,
  defined by \eqref{eq:H^0-transfer}.
  \begin{Exa}
    Let $G$ be a finite group. Then
    $BG$ is a classifying space of $G$, and
    $H^*(BG;\ZZ)$ is group cohomology 
    of $G$ with trivial coefficients.     
  \end{Exa}
  \begin{Exa}
    More generally, let $X$ be a finite $G$-set. Then
    $$B(G\ltimes X) \simeq EG\times_GX,$$
    and $H^*(EG\times_GX;\ZZ)$
    is the Borel equivariant cohomology of $X$.     
  \end{Exa}
\subsubsection{Representation rings}\label{Representation-rings-Sec} 
  We write $\on{Vect^{f}_\CC}$ for the category of finite dimensional
  complex vector spaces and consider the functor category
  \begin{eqnarray*}
    Rep_\CC(\uG) & = & \mF\!un\(\uG,\on{Vect^{f}_\CC}\)
  \end{eqnarray*}
  of complex $\uG$ representations.
  Let $\map\phi\uG\uH$ be a map of groupoids. Then there is an adjoint
  functor pair 
  \begin{eqnarray*}
     \ind\phi{}:Rep_\CC(\uG) & \leftrightharpoons & Rep_\CC(\uH) : \res{}\phi,
  \end{eqnarray*}
  where $\res{}\phi$ is precomposition with $\phi$, while
  $\ind\phi{}$ is left Kan extension along $\phi$. 
  We write
  $R(\uG)$ for the Grothendieck group of the category of
  representations
  of $\uG$.
  The maps induced by
  $\on{ind}{}{}$ and $\on{res}{}{}$ make $R$ into a global Mackey
  functor satisfying the surjectivity axiom\footnote{A good reference
    for the groupoid formulation of this fact is 
  \cite[{Prop.\ 0.0.1}]{Panchadcharam:Categories_of_Mackey_functors}.}.
  For a fixed groupoid $\uG$, we also have the 
  $\uG$-Mackey functor sending $\uH$ to the $H^2(B\uG;U(1))$-graded
  ring of projective representations,
  \begin{eqnarray*}
    R^*(\uH) & = & \bigoplus_\alpha R^{\phi^*(\alpha)}(\uH).
  \end{eqnarray*}
  \begin{Exa}
    If $\phi$ is an inclusion of finite groups, then $\ind\alpha{}$ is
    the familiar induced representation functor.     
  \end{Exa}
  \begin{Exa}
    If $\phi$ is
    the unique map from a finite group $G$ to the trivial group, then
    $\ind\phi{}$ is the inner product with one, i.e., $\ind\phi{}\rho$
    is the trivial summand of $\rho$.
  \end{Exa}
  \begin{Exa}
    Let $G$ be a finite group acting on a finite set $X$, and let
    $\uG=G\ltimes X$ be the corresponding translation groupoid. Then
    the category of $\uG$-representations is naturally identified with
    the category of $G$-equivariant vector
    bundles on $X$.
    Let $\map fXY$ be a 
    map of $G$-sets, and let
    $\phi:=G\ltimes f$, then
    $\ind\phi{}$ is the equivariant Atiyah transfer along $f$.     
  \end{Exa}
\subsubsection{$K(n)$-local cohomology theories} 
  In \cite{Strickland:K(N)-local}, Strickland showed that the object
  $B\uG$ becomes self-dual in the $K(n)$-local category $\mathcal S_{K(n)}$,
  a localization of the stable homotopy category $\mathcal S$.
  Let $E^*$ be a generalized cohomology theory that is well-defined on
  $\mathcal S_{K(n)}$, for instance, Morava-Lubin-Tate theory $E_n^*$. 
  Then Strickland's work implies that 
  \begin{eqnarray*}
    \uG& \longmapsto & E^*(B\uG)
  \end{eqnarray*}
  is a global Mackey functor with transfers along all maps.
  The surjections axiom does not hold in this example.
\subsubsection{n-Class functions}\label{n-Class-Sec}
If $M$ is a Mackey functor then so is $M\circ \mI$, where $\mI$
denotes the inertia groupoid. If follows that, for any $n\in
\mathbb N$, we have a Mackey functor
\begin{eqnarray*}
  \on{n-Class}(\uG,R) & := & H^0(B\mI^n\uG;R)
\end{eqnarray*}
with transfers along faithful maps (transfers along all maps if $R$ is
a $\QQ$-algebra).
An element $\chi\in \on{n-Class}(\uG,R)$ is called an {\em $n$-class
  function} on $\uG$. It is defined on $n$-tuples of commuting
automorphisms of $\uG$ and satisfies
\begin{eqnarray*}
  \chi(sg_1s\inv,\dots,sg_ns\inv) & = & \chi(g_1,\dots, g_n)
\end{eqnarray*}
(invariance under simultaneous conjugation). For $n\geq 2$ this does
not satisfy the surjectivity axiom. The group $Aut(\ZZ^n)$ acts
on $\mI^n\uG = Hom(\ZZ^n,\uG)$, inducing an action by Mackey-functor
automorphisms on $\on{n-Class}(-,R)$. We will also be interested in the Mackey functor
\begin{eqnarray*}
  \uG&\longmapsto& \on{n-Class}(\uG;R)^{Aut(\ZZ^n)}.
\end{eqnarray*}
\begin{Def}
  For a topological space $X$, we write $Prin_\uG(X)$ for the category
  of principal $\uG$-bundles over $X$ and their isomorphisms over $X$.
\end{Def}
If $X$ is the $n$-torus
\begin{eqnarray*}
  T^n & = & \SS^1\times\dots\times \SS^1
\end{eqnarray*}
then we have an $Aut(\ZZ^n)$-equivariant equivalence of groupoids
\begin{eqnarray}
  \label{eq:Prin-I}
  \mI^n(\uG) & \simeq & Prin_\uG(T^n), 
\end{eqnarray}
where the $Aut(\ZZ^n)$-action on the left-hand side is by
automorphisms of $T^n=B\ZZ^n$. So, $\on{n-Class,R}(\uG)$ may be
interpreted as a ring of global functions on $Prin_\uG(T^n)$, and
$\on{n-Class}(\uG)^{Aut(\ZZ^n)}$
as ring of $Aut(\ZZ^n)$-invariant such global functions.

A variation of these are the Mackey functors 
\begin{eqnarray*}
  \uG&\longmapsto & \on{n-Class_p}(\uG,R)
\\
 \text{and}\quad\quad\quad \uG&\longmapsto &
 \on{n-Class_p}(\uG,R)^{Aut((\ZZ\widehat{{}_p})^n)},
\end{eqnarray*}
of {\em $p$-adic $n$-class functions},
defined exactly like $\on{n-Class}$, but with the inertia
groupoid $\mI$ replaced by the $p$-adic inertia groupoid
\begin{eqnarray*}
  I_p(-) & = & \mF\!un(\ZZ\widehat{{}_p},-).
\end{eqnarray*}
So, elements of $\on{n-Class_p}(\uG,R)$ are functions defined on commuting
$n$-tuples of $p$-power order automorphisms and invariant under simultaneous conjugation.
\subsubsection{Subgroup class functions}
\begin{Def}
  Let $\uG$ be a finite groupoid. Then the {\em subgroup groupoid} of
  $\uG$, denoted $\mathcal S(\uG)$ has as objects the subgroups of $\uG$ and
  arrows
  $$%
    \xymatrix{H \ar[r]^{g\,\,\,\,\,\,\,\,} & gHg\inv.}
  $$
  A {\em subgroup class function} (a.k.a.\ {\em
    ``supercentral function''} in \cite{Knutson:lambda-rings})
  with values in $R$ is a function from $\mathcal S(\uG)$ to $R$. The
  ring of all subgroup class functions on $G$ is denoted
  \begin{eqnarray*}
    SCF(\uG,R) & = & R^\uG.
  \end{eqnarray*}
  If $R$ is clear from the context, we may simply write $SCF(\uG)$.
\end{Def}
\begin{Rem}
{The term ``supercentral function'' dates back to the time before
      ``super'' meant $\ZZ/2\ZZ$-graded.}
\end{Rem}
Similarly to $\on{n-Class}$, one sees that $SCF$ is a
Mackey functor with transfers along faithful maps. If $R$
contains $\QQ$, then SFC has all transfers, but no surjections axiom.
\subsubsection{Character maps}
\begin{Exa}
  We have a natural transformation of Mackey functors 
  \begin{eqnarray*}
    char\negmedspace : R(\uG) & \longrightarrow & \on{1-Class}(\uG;\QQ[\mu_{\infty}])
  \end{eqnarray*}  
  defining an isomorphism onto its image
  \begin{eqnarray*}
    char\negmedspace : R(\uG) &
    \xrightarrow{\phantom{x}\cong\phantom{x}} &
    \on{1-Class}(\uG;\ZZ[\mu_{\infty}])^{\widehat\ZZ}. 
  \end{eqnarray*}  
\end{Exa}
\begin{Exa}
  Let $E_n$ be the degree zero part of Morava-Lubin-Tate theory. Then
  we have a natural transformation of Mackey functors (with all
  transfers)
    \begin{eqnarray*}
    char\negmedspace :E_n(B\uG) & \longrightarrow & \on{n-Class_p}(\uG;
    L )
  \end{eqnarray*}  
  defining an isomorphism onto its image
  \begin{eqnarray*}
    char\negmedspace : E_n(B\uG) &
    \xrightarrow{\phantom{x}\cong\phantom{x}} &
    \on{n-Class_p}(\uG;L)^{Aut(\widehat{\ZZ}_p^n)},
  \end{eqnarray*}
  see \cite[Thms C\& D]{Hopkins:Kuhn:Ravenel} and \cite[Prop.1.6]{Ganter:Orbifold_genera}.
\end{Exa}

Recall that the Burnside ring $A$ is initial among the global Mackey
functors with transfers along faithful maps:
given a global Mackey functor $M$, the unique natural transformation of
Mackey functors $\eta\negmedspace : A\Rightarrow M$ is given by
\begin{eqnarray*}
  \eta_G\negmedspace : A(G) & \longrightarrow & M(G) \\
                      \left[G/H\right] & \longrightarrow & ind_H^G(1).
\end{eqnarray*}
If $M$ has all transfers {\em and the surjections axiom holds} then
$\eta$ preserves all transfers. We will see that if $M$ is a global Green functor 
(resp.\ global power functor), defined below, then $\eta$ is a
transformation of Green functors (resp.\ power functors).

The following two examples should be compared to \cite[Chapter
  2.4]{Knutson:lambda-rings} and \cite{Solomon:The_Burnside_algebra}.
\begin{Exa}[Subgroup character of a permutation representation]
  The transformation 
  \begin{eqnarray*}
    \eta^{SC}\negmedspace : A & \longrightarrow & \on{SCF}(-,\ZZ)
  \end{eqnarray*}
  sends the $G$-set $Y$ to the subgroup class function 
  \begin{eqnarray*}
    \chi_Y \negmedspace : [H] & \longmapsto & |Y^{H}|
  \end{eqnarray*}
  ($H$-fixed points). Knutson proves that $\eta$ is injective. The
  subgroup characters of the transitive $G$-sets are Burnside's {\em `marks'}.
\end{Exa}
\begin{Exa}[Chromatic character maps]
\label{exa:chromatic_characters}
  The transformation
  \begin{eqnarray*}
    \eta^{n}\negmedspace : A & \longrightarrow & \on{n-Class}(-,\ZZ)
  \end{eqnarray*}
  sends the $G$-set $X$ to the $n$-class function 
  \begin{eqnarray*}
    \chi_X \negmedspace :(g_1,\dots,g_n) & \longmapsto & |X^{(g_1,\dots,g_n)}|.
  \end{eqnarray*}
\end{Exa}  
The two last examples are related by the maps
\begin{eqnarray*}
  \mI^n(\uG) & \longrightarrow & \mathcal S(\uG)\\
      (g_1,\dots,g_n) & \longmapsto & \langle g_1,\dots, g_n\rangle, 
\end{eqnarray*}
inducing the transformations of Mackey functors
\begin{eqnarray*}
  SCF(\uG,\ZZ) & \longrightarrow & \on{n-class}(\uG,\ZZ)^{Aut(\widehat\ZZ^n)}.
\end{eqnarray*}
For $n=1$, we have the commuting diagram
$$%
  \xymatrix{
    A  \ar[rr]^{\eta\quad\quad\quad\quad}\ar[rd]_\eta & &
    \on{1-Class}(-,\ZZ[\mu_\infty]) \\
   & R \ar[ru]_{char}&
  }
$$
and $\chi_X$ is the character of the permutation representation
$G\to S_{|X|}$. Similarly, for $n=2$, we have the permutation
2-representation defined by $X$ with the trivial 2-cocycle, and its
2-character is $\chi_X$, see \cite{Ganter:Kapranov:Representation}.
In the groupoid picture, $\chi_X$ counts the size of the fibers of
the map
\begin{eqnarray*}
  [\mI^n(\uG\ltimes X)] & \longrightarrow & [ \mI^n(\uG)].
\end{eqnarray*}
\subsubsection{$G$-modular functions} 
\label{G-modular-Sec}
  Let $\mathfrak H$ be the upper half plane and write
  \begin{eqnarray*}
    \mM_\uG & = & \mI^2\uG\times_{SL_2\ZZ}\mathfrak H
  \end{eqnarray*}
  for the orbifold groupoid 
  $$
    SL_2\ZZ\ltimes \(\mI^2\uG\times \mathfrak H\).
  $$
  The space
  $\mM_\uG$ parametrizes
  principal $\uG$-bundles over complex elliptic curves
  (see \cite{Ganter:Hecke_operators},
  \cite{Morava:Moonshine_elements}, \cite{Carnahan:Generalized_MoonshineI} . Global
  functions on $\mM_\uG$ are functions $F(g,h;\tau)$ satisfying
  \begin{eqnarray}\label{eq:moonshine1}
    F\(g^ah^c,g^bh^d;\tau\) & = & F\(g,h;\frac{a\tau+b}{c\tau+d}\)
  \end{eqnarray}
  and
  \begin{eqnarray}\label{eq:moonshine2}
    F\(sgs\inv,shs\inv;\tau\) & = & F\(g,h;\tau\).
  \end{eqnarray}

  For $\alpha\in H^3(B\uG;U(1))$, we have the Freed-Quinn line bundle
  $\mL^\alpha$ over $\mM_\uG$, see \cite{Freed:Quinn}.
  Sections of $mL^{\alpha}$ may be interpreted as
  functions $F(g,h;\tau)$ as above satisfying \eqref{eq:moonshine1}
  up to a root of unity and satisfying \eqref{eq:moonshine2}, see
  \cite{Ganter:Hecke_operators}. We write
  \begin{eqnarray*}
    Mod(\uG) & = & \Gamma\(\mM_\uG,\mathscr O_{\mM_G}\)
  \end{eqnarray*}
  for the Mackey functor of {\em $\uG$-modular functions} as in
  \eqref{eq:moonshine1} and \eqref{eq:moonshine2}. Its transfer along
  $\phi$ is given by the formula \eqref{eq:H^0-transfer} applied to $\mI^2\phi$.
  For fixed $\uG$, we also define the 
  $\uG$-Mackey functor sending $\map\phi\uH\uG$ to the
  $H^3(BG;U(1))$-graded ring
  \begin{eqnarray*}
    Mod^*(\uH) & = & \bigoplus_{\alpha}\Gamma\(\mM_\uH,\mL^{\phi^*(\alpha)}\).
  \end{eqnarray*}
  A level $2$ version of the Mackey functor $Mod(\uG)$ was studied in
  \cite{Devoto:Equivariant_elliptic_homology}.
\subsubsection{Tate $K$-groups}\label{sec:q-exp}
Let $\xi$ be the natural automorphism of $id_{\mI\uG}$ with
\begin{eqnarray*}
  \xi_{(x,g)} & = & g.
\end{eqnarray*}
The Mackey functor 
\begin{eqnarray*}
  K_{Tate}(\uG) & \subset & R(\mI\uG)
\end{eqnarray*}
considered in \cite{Ganter:Orbifold_K-Tate}
is the Grothendieck group of formal power series 
$\sum V_nq^\frac n{|\xi|}$ such that $\xi$ acts on $V_n$ with eigenvalue $e^{2\pi
  in/|\xi|}$.
For fixed $\uG$, the transgression map
\begin{eqnarray*}
  \tau\negmedspace : H^3(B\uG;U(1)) & \longrightarrow & H^2(B\mI\uG;U(1)),
\end{eqnarray*}
defined, for instance, in \cite{Willerton:Twisted_Drinfeld}, gives rise
to the twisted Tate $K$-groups
\begin{eqnarray*}
  K^\alpha_{Tate}(\uG) & \subset & R^{\tau(\alpha)}(\mI\uG)\ps{q^{\frac1{h|\xi|}}}
\end{eqnarray*}
as follows: the cocycle $\tau(\alpha)$ classifies an extension
$\widetilde{\mI\uG}$ of $\mI\uG$. A lift $\widetilde\xi$ of $\xi$ in
$\widetilde{\mI\uG}$ has order $h|\xi|$ for some natural number $h\geq
1$. Then $K^\alpha_{Tate}$ is the Grothendieck group of formal power
series
$\sum V_nq^{\frac1{h|\xi|}}$ such that 
$$V_n\in
Rep_\CC(\widetilde{\mI\uG})$$ is an 
eigenspace of $\widetilde\xi$ 
with eigenvalue $e^{2\pi i n/h|\xi|}$.
This gives rise to an $H^3(B\uG;U(1))$-graded $\uG$-Mackey functor
$K^*_{Tate}$.

\medskip
By the {\em $q$-expansion principle} of \cite[Sec.5]{Freed:Quinn}
(see also \cite[2.4]{Ganter:Hecke_operators}), we have a natural transformation of
Mackey functors
\begin{eqnarray*}
  Mod(-) & \Longrightarrow & K_{Tate}(-)\tensor \CC
\end{eqnarray*}
and of $H^3(B\uG;U(1))$-graded $\uG$-Mackey functors
\begin{eqnarray*}
  Mod^*(-) & \Longrightarrow & K^*_{Tate}(-)\tensor \CC.
\end{eqnarray*}
Alternatively, we could interpret $q$-expansion as a map
\begin{eqnarray*}
  Mod^\alpha(\uG) & \longrightarrow &
  \bigoplus_{[g]}\on{1-Class}\(aut_{\widetilde{\mI\uG}}(g),\ZZ[\mu_\infty]\)\ps{q^\frac1{h|g|}}
\end{eqnarray*}
and note that if the modular section $F$ is invariant under the $Aut(\widehat\ZZ)$
action on each summand of the target, then the $q$-expansion actually
takes values in $K_{Tate}^\alpha(\uG)$. This is the second condition of Norton's
generalized Moonshine conjecture \cite{Norton:Generalized_moonshine}.
\section{Products}
\subsection{Green functors} 
Assume that the additive category $\mA$ comes equipped with a bilinear
symmetric
monoidal structure $\tensor$ with unit $A$.
\begin{Def}
  A {\em globally defined Green functor} with values in
  $(\mA,\tensor,A)$ is a globally defined Mackey functor $M$, satisfying $M(\pt)
  = \mA$ and such that the contravariant part, $M^*$, takes 
  values in the category of commutative ring objects in $\mA$.
\end{Def}
\begin{Prop}
  A globally defined Green functor with values in
  $(\mA,\tensor,A)$ is the same thing as a globally defined Mackey
  functor $M$ together 
  with an isomorphism $M(\pt)\cong A$ and with unitary, commutative and
  associative {\em K\"unneth maps}
  \begin{eqnarray}\label{eq:mu}
    \kappa\negmedspace : M(\uG)\tensor M(\uH) & \longrightarrow & M(\uG\times\uH),
  \end{eqnarray}
  making $M^*$ a (lax) bimonoidal functor.
\end{Prop}
\begin{Pf}{}
  Let $M$ be a global Mackey functor. Then the K\"unneth map
  \eqref{eq:mu} is defined as the
  composite of
  $$
    M(\uG)\tensor M(\uH)
    \xrightarrow{\phantom{XX}pr_\uG^*\tensor pr_\uH^*\phantom{XX}} 
    M(\uG\times\uH)\tensor M(\uG\times\uH)
  $$
  with the multiplication of the ring object
  $M(\uG\times\uH)$. Conversely, given K\"unneth maps, the
  multiplication on $M(\uG)$ is given by
  \begin{eqnarray*}
    \delta^*\circ\kappa\negmedspace :M(-)\tensor M(-) & \longrightarrow & M(-)
  \end{eqnarray*}
  and the unit by 
  \begin{eqnarray*}
    \varepsilon^*\negmedspace : A& \longrightarrow & M(-).
  \end{eqnarray*}
  Here $\delta$ and $\varepsilon$ are the natural transformations on
  $Gpd$ sending $\uG$ to its diagonal map
  \begin{eqnarray*}
    \delta_\uG\negmedspace : \uG&\longrightarrow & \uG\times \uG
  \end{eqnarray*}
  and to the unique map
  \begin{eqnarray*}
    \varepsilon_\uG\negmedspace : \uG&\longrightarrow&\pt.
  \end{eqnarray*}
\end{Pf}
In almost all of our examples, the K\"unneth map $\kappa$ turns out to be
an isomorphism. The exception is the Burnside ring example, and in the
$K_{Tate}$-case, we need to invert $q$ to have a K\"unneth isomorphism.
\subsubsection{Consequences of the definition}
In all our examples of Green functors, it makes sense to speak of
elements of objects of $\mA$, so, to simplify notation, we will denote
the product of two elements by $s\cdot t$. 
Let $M$ be a global Green functor. Then the
following hold.
\begin{description}
\item[Frobenius axiom:] Let $\map\phi\uG\uH$ be a faithful map of
  groupoids. For elements $s\in M(\uG)$ and $t\in M(\uH)$, we have
  \begin{eqnarray*}
    \phi_*(s)\cdot t & = & \phi_*\(s\cdot\phi^*t\)
  \end{eqnarray*}
  and 
  \begin{eqnarray*}
    t\cdot \phi_*s & = & \phi_*\((\phi^*t)\cdot s\).
  \end{eqnarray*}
  \begin{proof}
    The second equation follows from naturality of $\kappa$ and the homotopy pullback square
    $$
      \xymatrix@=2.7ex{
        \uG\ar[2,0]_{\(\phi\times\id_\uG\)\circ\delta_\uG}\ar[0,3]^{\phi}&&&
        \uH\ar[2,0]^{\delta_\uH}\\  \\
        \uH\times\uG\ar[0,3]^{\id_\uH\times\phi}&&&\uH\times\uH.
      }
    $$
  The first equation is proved analogously.
  \end{proof}
  \item[Inner products:]
    We have a bilinear form
    \begin{eqnarray*}
      \langle r,s\rangle_\uG & = & \varepsilon_{\uG,*}(r\cdot s)
    \end{eqnarray*}
    on $M(\uG)$.
  \item[Frobenius reciprocity:] Let $\map\phi\uG\uH$ be a faithful
    map. Then we have
    \begin{eqnarray*}
      \langle r,\phi_*s\rangle_\uH  & = & \langle \phi^*r,s\rangle_\uG     
    \end{eqnarray*}
    
\end{description}

\subsection{The graded rings $S_M(G)$}
\begin{Def}
  The $n$th symmetric power of a groupoid $\uG$ is the groupoid
  $S_n\smallint \uG$ obtained from $\uG^n$ by adding arrows
  \begin{eqnarray*}
    (x_1,\dots,x_n) & \xrightarrow{\phantom{x}\sigma\phantom{x}} &
    (x_{\sigma\inv(1)},\dots,x_{\sigma\inv(n)})
  \end{eqnarray*}
  for $\sigma\in S_n$ and $(x_1,\dots,x_n)\in G_0^n$, with the
  relation
  \begin{eqnarray*}
    \sigma (g_{1},\dots,g_{n}) & = &   (g_{\sigma\inv(1)},\dots,g_{\sigma\inv(n)})\sigma.
  \end{eqnarray*}
\end{Def}
This construction is functorial in $\uG$ and preserves
equivalences. 
If $\uG = G$ is a group then $S_n\smallint G$ is the wreath product of
$S_n$ with $G$.
We have 
\begin{eqnarray*}
  S_0\smallint \uG & = & 1
\end{eqnarray*}
and equivalences
\begin{eqnarray}\label{eq:S_nxS_m}
   \coprod_{k=0}^n(S_k\smallint\uG)\times(S_{n-k}\smallint\uH) &
   \xrightarrow{\phantom{xx}\sim\phantom{xx}} & 
   S_n(\uG\sqcup\uH).
\end{eqnarray}
Let $M$ be a global Green functor, and let $\uG$ be a finite
groupoid. 
\begin{Def}
  We write
  \begin{eqnarray*}
    S_M(\uG)&:=&{\bigoplus^\wedge_{n\geq 0}}\, M(S_n\lwr \uG)\,t^n, 
  \end{eqnarray*}
  for the infinite product of the $M(S_n\lwr\uG)$.
  Here $t$ is a dummy variable, keeping track of the grading.
  We will also write $S_M$ for $S_M(\pt)$.
\end{Def}
$S_M$ is bivariantly functorial, and \eqref{eq:mu} and
\eqref{eq:S_nxS_m} give a natural transformation
\begin{eqnarray*}
  \kappa\negmedspace :S_M(\uG)\otimes S_M(\uH) & \longrightarrow &
  S_M(\uG\sqcup\uH).
\end{eqnarray*}
There are two ring structures on $S_M(\uG)$, namely
\begin{description}
\item[Juxtaproduct:] This is defined by degree-wise multiplication,
  using the ring structure of $M(S_n\lwr\uG)$. The juxtaproduct of two
  elements of different degrees is zero. The unit of the juxtaproduct
  is $\mathbf 1 = \sum 1_n\,t^n$, where $1_n$ stands for $1\in
  M(S_n\lwr\uG)$.
\item[Cross product:] 
  The cross product is
  $$
    \times\negmedspace : S_M(\uG)\tensor S_M(\uG)
    \xrightarrow{\phantom{xx}\kappa\phantom{xx}} 
    S_M(\uG\sqcup\uG) \xrightarrow{\phantom{xx}S_{M_*}(f)\phantom{xx}} 
    S_M(\uG),
  $$
where $\map f{\uG\sqcup\uG}\uG$ is the fold map. This makes $S_M(\uG)$
into a graded ring with unit $1=1_0$.
\end{description}
We will see below that $(S_M,\times,1_0)$ may be viewed as a ring of
internal operations on $M(\uG)$, if $M$ is a global power functor with
all transfers.

The definition of $S_M$ still makes sense if $M$ is replaced by an arbitrary
bivariant functor out of $Gpd$. So, we may iterate $S$, and
find that 
\begin{eqnarray*}
  M   &\longmapsto&   S_M
\end{eqnarray*}
is a comonad, whose comultiplication are
\begin{eqnarray*}
  \nabla^*\negmedspace  :S_M&\Longrightarrow & S_{S_M}
\end{eqnarray*}
and counit
\begin{eqnarray*}
  \upsilon \negmedspace: S_M&\Longrightarrow & M
\end{eqnarray*}
where $\nabla$ is made up of the canonical inclusions
\begin{eqnarray*}
  S_n\lwr S_m\lwr\uG&\quad\hookrightarrow\quad & S_{nm}\lwr\uG
\end{eqnarray*}
and $\upsilon$ is projection onto the coefficient of $t^1$.
If
$$
  \longmap\kappa{M(S_n\lwr\uG)\tensor M(S_m\lwr\uG)}M((S_n\lwr\uG)\times (S_m\lwr\uG))
$$
is an isomorphism for all $m,n$, then
$S_M(\uG)$ is a Hopf algebra over $M(\pt)$ whose
coproduct is the pullback along \eqref{eq:S_nxS_m} (see \cite[Thm 1.2]{Hoffman:tau-rings}).
\subsubsection{The elements $f_n$ and $c_n$}\label{sec:elements}
Let $M$ be a global Green functor, $\mathbf 1=\sum 1_nt^n$ the unit of the
juxtaproduct on $S_M$. There are two important sequences of elements
of $S_M$. Let 
\begin{eqnarray*}
  f(t) & = & \sum_{n\geq 0}f_n\,t^n 
\end{eqnarray*}
be the power series with 
\begin{eqnarray*}
  \mathbf 1 \times f(-t) & = & 1.
\end{eqnarray*}
This determines the $f_n\in M(S_n)$ inductively. The elements $c_k\in
M(S_k)$ for $k\geq 1$ are defined by
\begin{eqnarray*}
  \sum_{k\geq 1} c_k\,t^{k-1} & = & \frac d{dt}\,ln\,\mathbf 1,
\end{eqnarray*}
where the logarithmic power series $ln$ is computed using the cross product.
\subsection{Examples}
\subsubsection{Burnside rings}
The Burnside ring functor is a global Green functor with the K\"unneth
map given by the Cartan product of sets. The ring 
\begin{eqnarray*}
  \mathbf B & = & \bigoplus^\wedge_{n\geq 0}A(S_n)
\end{eqnarray*}
(read `Beta') is the free $\beta$-ring on one generator. It makes an
early brief appearance in \cite{Knutson:lambda-rings}, and plays a
central role in \cite{Rymer:Power_operations}, 
\cite{Ochoa:Outer_plethysm}, \cite{Morris:Wensley:Adams_operations},
\cite{Vallejo:The_free_beta-ring},  
\cite{Guillot:Adams_operations}, and \cite{Morava:Santhanam}.

Let $H\sub S_n$ be a subgroup. Then evaluation at $H$ is an element of
$SCF(S_n)^*$. Precomposing with $\eta^{SCF}$, this becomes the
element of $A(S_n)^*$.
Similarly, $H$ defines the maps 
\begin{eqnarray}
\notag
  SCF(S_n\times \uG)&\longrightarrow& SCF(\uG)\\
\notag
  \chi &\longmapsto& \chi([H]\times-)\\
&\text{and}&\\
\label{eq:psi_H}
\notag
  A(S_n\times \uG)&\longrightarrow& A(\uG)\\
\notag
  Y &\longmapsto& Y^H.
\end{eqnarray}
%
\subsubsection{Linear and projective representations}
\label{sec:linear-and-projective} 
If $M=R$ is the representation ring functor then the Schur-Weyl theorem
gives the Hopf algebra isomorphism
\begin{eqnarray*}
  \bigoplus_{n\geq0}R(S_n) & \cong & \on{lim}\,\,\ZZ[x_1,\dots,x_m]^{S_m}
\end{eqnarray*}
identifying the representation ring of the symmetric groups with the
Hopf ring of symmetric functions in infinitely many variables.
On the left-hand side, the element $f_n\in R(S_n)$ is the alternating representation, while
$c_k$ is the element satisfying
\begin{eqnarray*}
  \langle\varrho, c_k\rangle_{S_k} & = & \chi_\varrho((1,\dots,k)).
\end{eqnarray*}
On the right-hand side, $f_n$ corresponds to the $n$th elementary
symmetric function, while $c_k$ becomes the $k$th power sum function.

We note that, differing from the standard convention, the inner product
on $R(\uG)$ is  
\begin{eqnarray*}
  \langle\varrho_1,\varrho_2\rangle & = & \(\varrho_1\tensor\varrho_2\)^G.
\end{eqnarray*}

\medskip
Let $\alpha_n$ be the cocycle classifying Schur's spin extension 
$$
  1\longrightarrow \{\pm1\}\longrightarrow \widetilde
  S_n\longrightarrow S_n\longrightarrow 1
$$
For $n>3$ Schur showed that $[\alpha_n]$ is the only non-trivial
element of $H^2(S_n;U(1))$.
Following J\'ozefiak, we write $R^-(S_n)$ for the Grothendieck group of
projective super-representations of $S_n$ with cocycle $\alpha_n$ and
$\Gamma = \ZZ[p_1(X), p_3(X),\dots]$ for the graded ring of
polynomials in odd power sums of the variables $x_1, x_2, \dots$ with
$deg(p_j(X)) = j$. 
Following Schur \cite{Schur:Ueber_die_Darstellung}, 
J\'ozefiak constructs an isomorphism of graded $\QQ[\sqrt2]$-algebras
\begin{eqnarray*}
  \bigoplus_{n\geq 0}R^-(S_n)_{\QQ[\sqrt2]}&
  \xrightarrow{\phantom{xx}\cong\phantom{xx}} &
  \Gamma\tensor \QQ[\sqrt2].
\end{eqnarray*}
\subsubsection{Morava-Lubin-Tate theories}
The Hopf ring
\begin{eqnarray*}
    S_{E_n} & = & \bigoplus^\wedge_{n\geq 0} E_n(BS_n)
\end{eqnarray*}
is the subject of \cite{Strickland:Morava_E-theory} and
\cite{Strickland:Turner}.
\subsubsection{n-Class functions}
Recall that $n$-class functions are functions on 
\begin{eqnarray*}
  [\mI^n\uG] & \cong & Prin_\uG(T^n) /\cong,  
\end{eqnarray*}
and that $Prin_{S_k}(T^n)$ is equivalent to the category of $k$-fold
covers of $T^n$. 
First, we consider the case $\uG=\pt$.
The infinite groupoid
\begin{eqnarray*}
  \coprod_{k\geq 0}\, \mI^n S_k & \cong & \on{Cov^{fin}}(T^n) 
\end{eqnarray*}
is a monoid in $G\!pd$, with the
concatenation product on the left-hand side corresponding to disjoint
union of covers on the right. So, indecomposable covers are the
connected covers. Indecomposable $n$-tuples of permutations are
described as follows: let $A$ be an abelian group together with a
surjective group homomorphism
$$
  \xymatrix{
  \alpha\negmedspace : \ZZ^n \ar@{->>}[r]& A.
  }
$$
Any numbering of the elements of $A$ yields a map
$$
  \xymatrix{
  A \ar@{>->}[r]& S_{|A|},
  }
$$
which we compose with $\alpha$ to obtain an isomorphism class
\begin{eqnarray*}
  [\alpha] & \in & [\mI^nS_{|A|}],
\end{eqnarray*}
independent of our choice of numbering. Elements of this form are
exactly the indecomposable objects on the left.
It follows that we have an isomorphism of graded rings
\begin{eqnarray*}
  S_{\on{nClass_R}} & \cong &
  R[\chi_\alpha']_{\alpha\in\mathfrak A'}
\end{eqnarray*}
where $\chi_\alpha'$ is the characteristic function of $[\alpha]$, with
$deg(\chi'_\alpha)  =  |A| $, while
\begin{eqnarray*}
  \mathfrak A' & = & \{\alpha:\ZZ^n\twoheadrightarrow A\}/\cong
\end{eqnarray*}
is the set of epimorphisms out of $\ZZ^n$ up to isomorphisms under
$\ZZ^n$. Similarly,
\begin{eqnarray*}  
  \(S_{\on{nClass_R}}\)^{Aut(\ZZ^n)} & \cong & R[\![\chi_\alpha]\!]_{\alpha\in\mathfrak A},
\end{eqnarray*}
where
function $\chi_\alpha$ is the characteristic function of the
$Aut(\ZZ^n)$-orbit of $[\alpha]$, and
\begin{eqnarray*}
  \mathfrak A & = & \mathfrak A'/Aut(\ZZ^n).
\end{eqnarray*}
Let now $\uG$ be arbitrary.
\begin{Def}
  Let $X$ be a topological space.
  The category $\on{G-Cov_k}(X)$ (respectively $\on{G-Cov^{fin}}(X)$)
  has as objects composites
  $$
    \xymatrix{
    P\ar[d]_{\pi}\\
    Y\ar[r]^{\phi}& X,
    }
  $$
  where $\pi$ is a principal
  $G$-bundle and $\phi$ is an $k$-fold (respectively finite) cover.
  Morphisms in $\on{G-Cov_k}(X)$ are $G$-equivariant homeomorphisms $\map fPP'$
  covering the identity on $X$. 
\end{Def}
For any such morphism $f$, the map $\map{(f/G)}{Y}{Y'}$ is an
  isomorphism of $k$-fold covers of $X$.
\begin{Lem}\label{lem:Cov(G)}\label{axioms-Lem}
(a)  There is an equivalence of categories 
  \begin{eqnarray*}
    \on{Prin}_{S_k\lwr G}(X) & \simeq & \on{G-Cov}_{k}(X). 
  \end{eqnarray*}

\medskip
(b) Under these identifications, taking
disjoint union of covers of degrees $k$ and $m$ corresponds to
induction along the inclusion 
\begin{eqnarray*}
  S_k\times S_m & \hookrightarrow& S_{k+m}.
\end{eqnarray*}

\medskip
(c) Principal $S_k\lwr S_m\lwr G$-bundles are identified with composites
$$
  P\stackrel\pi\longrightarrow  Z\stackrel\chi\longrightarrow
  Y\stackrel\phi\longrightarrow X,
$$
where $\pi$ is a principal $G$-bundle, $\chi$ is a cover or degree $m$
and $\phi$ is a cover of degree $k$. 
Induction along the inclusion
\begin{eqnarray*}
  \nabla\negmedspace:{S_k\lwr S_m\lwr G}& \longrightarrow& {S_{km}\lwr G}  
\end{eqnarray*}
corresponds to the functor that composes $\chi$ and $\phi$ to obtain
an $km$-fold cover.

%
\end{Lem}
\begin{Pf}{of Lemma {\ref{lem:Cov(G)}}}
Fix $X$, and
assume we are given an $n$-fold cover 
$\map {\phi}{Y}X$ and a principal $G$-bundle $\map{\pi}{P}{Y}$ over
the total space of $\phi$.  
We let $Q$ be the bundle whose fiber over $x$ consists of  all
possible numberings of the fiber of $Y$ over $x$. Then $Q$
is a principal $\Sn$-bundle over $X$ whose associated $n$-fold cover
is canonically isomorphic to $Y$.
Similarly, we define a bundle $R$ over $Q$ whose fiber over
$(y_1,\dots, y_n)$ consists of ordered $n$-tuples $(p_1,\dots, p_n)$
of points in $P$ satisfying
${\pi}(p_i) = y_i$. The action of $G$ on $P$
makes $R$ into a principal $G^n$-bundle
over $Q$, and the action of $\Sn$ on the fibers of $R$ makes $R$ into a
principal $S_n\lwr G$-bundle over $X$.

Conversely, assume that
$
  \map\psi RX
$
is a principal
$\Sn\lwr G$-bundle.
Let $Q:= G^n\backslash R$, and consider the commuting diagram
$$
  \xymatrix{
    R\ar[d]_{\xi}\ar[rd]^{\psi}\\
    Q\ar[r]_{\eta}&X,
  } 
$$
where $\xi$ and $\eta$ are the quotient maps. Then $\xi$ naturally has the
structure of a 
principal $G^n$-bundle, 
while $\eta$ inherits that of a principal $S_n$-bundle. 

We view $S_{n-1}$ as a subgroup of 
$S_n$ by identifying the elements of $S_{n-1}$ with the permutations of
$\{1,\dots,n\}$ that fix $n$. 
Let
$$
  P:= \(S_{n-1}\lwr G\)\backslash R
$$
and
$$
  Y := S_{n-1}\medspace\backslash\medspace Q,
$$
and let $\pi$ and $\phi$ be the maps induced by $\xi$ and $\eta$.
Then $\map{\phi}{Y}{X}$ is the $n$-fold cover
associated to $\eta$. We need to show that $\map\pi PY$ is a principal
$G$-bundle. 
If $p\in P$ is the equivalence class of $r\in R$,
we let $gp$ be the equivalence class of
$$
  {(\id;1,\dots,1,g)\cdot r}.
$$
This action of $G$ on $P$ is well-defined, because
the elements of $\{1\}^{n-1}\times G$ commute with those of $S_{n-1}\lwr G$
in $S_n\lwr G$. 
To identify the fibers of $P$ over $Y$ and over $X$, we note that we
have a bijective map 
\begin{eqnarray*}
(S_{n-1}\lwr G)\medspace\backslash\medspace(\Sn\lwr G)
&\longleftrightarrow&
(S_{n-1}\backslash S_n)\times G\\
\overline{(\sigma;g_1,\dots,g_n)}&\longmapsto &
(\overline{\sigma},g_n). 
\end{eqnarray*}
Here $\overline{x}$ denotes the coset of $x$.
We conclude that $\pi$ is indeed a principal $G$-bundle.
It is straight-forward that the composites of the two functors
$$R\mapsto(P,Y)\quad\text{and}\quad
(P,Y)\mapsto R$$
are naturally isomorphic to the respective identity functors.
\end{Pf}{}
\begin{Cor}
  We have an equivalence of groupoids
  \begin{eqnarray*}
    \coprod_{k\geq 0}\mI^n(S_k\lwr \uG) & \simeq & \on{\uG-Cov^{fin}}.
  \end{eqnarray*}
  The graded algebra $S_{\on{n-Class}}(\uG)$ is the
  algebra of invariant functions on $\on{\uG-Cov^{fin}}$. The element $c_k$ is the
  function
  \begin{eqnarray*}
    c_k (Y\xrightarrow\phi T^n) & = &
    \begin{cases}
      1&\text{if $Y$ is connected}\\0&\text{else.}
    \end{cases}
  \end{eqnarray*}
  The element $f_k$ is the function
  \begin{eqnarray*}
    f_k (Y\xrightarrow\phi T^n) & = &
    \prod_{Y=\coprod Y_i}(-1)^{k_i-1},
  \end{eqnarray*}
  where the product is over the connected components of $Y$ and
  $\phi\at {Y_i}$ is a $k_i$-fold cover of $T^n$.
\end{Cor}
\subsubsection{Product decompositions}
\begin{Def}
  Let $\mI(\uG)[\xi^{\frac1k}]$
  be the groupoid obtained from $\mI(\uG)$ by adding the additional arrows
  $$
    (x,g)\xrightarrow{\phantom{xx}g^{\frac1k}\phantom{xx}} (x,g),
  $$
  subject to the relations 
  \begin{eqnarray*}
    (g^{\frac1k})^k = g  & \quad\text{ and }\quad & hg^\frac1k = (h\inv gh)^{\frac1k}h.
  \end{eqnarray*}  
\end{Def}
Then we have the equivalence of groupoids
\begin{eqnarray}
  \label{eq:cycle_decomposition}
  \coprod_{n\geq 0}\mI(S_n\lwr\uG) &
  \xleftarrow{\phantom{xx}\sim\phantom{xx}} &
  \prod_{k\geq1}\(\coprod_{m\geq0}S_m\lwr\(\mI(\uG)[\xi^\frac1k]\)\)
\end{eqnarray}
discussed, for instance in \cite{Ganter:Orbifold_K-Tate} (following
\cite{Dijkgraaf:Moore:Verlinde:Verlinde}).
As an immediate consequence of \eqref{eq:cycle_decomposition}, we obtain
the well known identity
\begin{eqnarray*}
  \sum_{n\geq0}\dim R(S_n)_\CC\,\, t^n & = & \prod_{k\geq 1}\frac1{1-t^k}.
\end{eqnarray*}
One step higher, we obtain a bivariantly natural isomorphism
\begin{eqnarray*}
  S_{2Class}(\uG) & \cong & \bigotimes_{k\geq 1}S_{1Class}\(\mI(\uG)[\xi^\frac1k]\)
\end{eqnarray*}
and, in particular,
\begin{eqnarray*}
  S_{2Class}(\pt) & \cong & \bigotimes_{k\geq 1}S_{1Class}\(\ZZ/k\ZZ\).
\end{eqnarray*}
A similar formula holds for Tate $K$-theory
\cite{Ganter:Orbifold_K-Tate}, but
the product decomposition \eqref{eq:cycle_decomposition} is
not compatible with the $SL_2\ZZ$-action.
\medskip

\noindent
{\bf Question:} 
Let
\begin{eqnarray*}
  \alpha_n & = &  p_1(\varrho_n)
\end{eqnarray*}
be the
first Pontrjagin class of the permutation representation.
It would be interesting to know whether there is a product
decomposition for 
\begin{eqnarray*}
  S_{K_{\!Tate}}^- & = & \bigoplus_{n\geq 0}^\wedge K_{\!Tate}^{\alpha_n}(S_n)
\end{eqnarray*}
in terms of the $S^\pm_R(\ZZ/k\ZZ)$ of Section
\ref{sec:linear-and-projective}. 
\subsubsection{$G$-modular functions}
In the case of $G$-modular functions, we have the graded rings
\begin{eqnarray*}
  S_{Mod}& = & \bigoplus^\wedge_{n\geq 0}\Gamma\(\mM_{S_n},\mathscr O_{\mM_{S_n}}\)
\end{eqnarray*}
and
\begin{eqnarray*}
  S^-_{Mod}& = & \bigoplus^\wedge_{n\geq 0}\Gamma\(\mM_{S_n},\mL^{p_1(\varrho_n)}\).
\end{eqnarray*}
Recall that $\mM_{S_n}$ 
is the moduli space parametrizing principal $S_n$-bundles over complex
elliptic curves or, equivalently, $n$-fold covers of complex elliptic curves.
\begin{Exa}\label{exa:tau+b/d}
  Let $\Gamma\sub\langle\tau,1\rangle$ be the sublattice generated by
  $d$ and $\tau+b$ with $0\leq b<d$. Then the degree $d$-cover
  \begin{eqnarray*}
    \phi\negmedspace :\CC/\Gamma & \longrightarrow & \CC/\langle\tau,1\rangle
  \end{eqnarray*}
  has $S_d$-monodromy $\varsigma_d$ along the circle $[1,0]$ and
  $\varsigma^b$ along $[0,\tau]$. The source of $\phi$ is isomorphic
  to $\CC/\langle\tau',1\rangle$ with $\tau'=(\tau+b)/d$.
\end{Exa}
\begin{Exa}\label{exa:atau}
  The degree $a$ isogeny
  \begin{eqnarray*}
    \CC/\langle a\tau,1\rangle & \longrightarrow & \CC/\langle\tau,1\rangle
  \end{eqnarray*}
  has $S_a$-monodromy $1$ along $[1,0]$ and $\varsigma_a$ along $[0,\tau]$.
\end{Exa}
Write
  \begin{eqnarray*}
    \mM_{S_n\lwr G} & \cong & \widetilde{\mathcal
      H}_{G,n}\,\sqcup\,\mM^{dec}_{G,n},
  \end{eqnarray*}
  where $\widetilde{\mathcal H}_{G,n}$ parametrizes the covers with
  connected total space and 
  $\coprod\mM^{dec}_{G,n}$ are the decomposable covers.
Then the function $c_n$ is 
\begin{eqnarray*}
      c_n\at{\widetilde H_n} \,\, \equiv \,\, 1 & \quad\text{ and }\quad & 
      c_n\at{\mM^{dec}_n} \,\,\equiv \,\, 0.
\end{eqnarray*}
It is believed that $S_{Mod}^-$ has more interesting elements than the
untwisted $S_{Mod}$. For instance, it seems reasonable to expect
transfers of the generalized Moonshine functions to turn up as
elements of $S_{Mod}^-$.
\section{Global power functors}
\subsection{Definition}
Let $A$ be a commutative ring with $1$, and let $M$  be a 
global Green functor with values in $\on{A-mod}$. 
\begin{Def}\label{power-operations-Def}
  A (total) power operation on $M$ is a bivariantly natural, non-additive,
  transformation 
  \begin{eqnarray*}
    P\negmedspace : M&\Longrightarrow & S_M
  \end{eqnarray*}
  satisfying
  \begin{description}
  \item[Exponentiality:] The map
    \begin{eqnarray*}
      P_\emptyset\negmedspace : M(\emptyset) & \longrightarrow &
      M(\pt) 
    \end{eqnarray*}
    maps $0$ to $1\in A$, and 
    \begin{eqnarray*}
      P_{\uG\sqcup\uH}\negmedspace : M(\uG)\oplus M(\uH) &
      \longrightarrow & S_M(\uG\sqcup\uH)
    \end{eqnarray*}
    maps $(a,b)$ to $\kappa(P_\uG(a)\tensor P_\uH(b))$.
  \item[Comodule axiom:] We have
    \begin{eqnarray*}
      P\circ P = \nabla\circ P\negmedspace :M&\Longrightarrow & S_{S_M}
    \end{eqnarray*}
    and 
    \begin{eqnarray*}
      \nu\circ P & = & \id_M.
    \end{eqnarray*}
  \end{description}
  A {\em global power functor} is a global Green functor $M$ together with a
  total power operation $P$.
\end{Def} 
\subsubsection{Consequences}
Writing
\begin{eqnarray*}
  P &=& \sum_{n=0}^{\infty}P_n\,t^n,
\end{eqnarray*}
we have
\begin{description}
  \item[Cartan formula:] $P(0) = 1_0$ and $P(a+b) = P(a)\times
    P(b)$ (cross product). 
  \item[Low degrees:] $P_1(x)=x$ and $P_0(x) = 1$. 
  \item[External product:] 
  $$
    \kappa(P_j(x)\tensor P_k(y)) = \res{S_{j+k}}{S_j\times S_k}(P_{j+k})(x).
  $$
  \item[Composition:]
  $$
    P_j(P_k(x)) = \res{S_{jk}}{S_j\smalllwr S_k}(P_{jk}(x))
  $$
  \item[Multiplicativity:] 
    \begin{eqnarray*}
      P_j(1) = 1&\quad\text{ and }\quad& P_j(ab) = P_j(a) P_j(b)    
    \end{eqnarray*}
  (juxtaproduct).
  \item[External products:]
  Let
    $$\map\delta{S_j\lwr(\uG\times\uH)}{(S_j\lwr\uG)\times(\Sj\lwr\uH)}$$
    be the map induced by the diagonal inclusion of $\Sj$ in
    $\Sj\times\Sj$. Then we have
    $$
      P_j(\kappa(x\tensor y)) = \res{}\delta\medspace\kappa(P_j(x)\tensor P_j(y)).
    $$
\end{description}
%
\subsection{$\tau$-rings and $\lambda$-rings}
Let $M$ be a global power functor whose K\"unneth maps are
isomorphisms. Following Hoffman \cite{Hoffman:tau-rings}, 
we define the total $\tau$-operation
$$
  \tau\negmedspace : M\uG\xrightarrow{\phantom{xxx}P\phantom{xxx}} S_M(\uG)
  \xrightarrow{\phantom{xxx}\delta^*\phantom{xxx}} S_M\tensor M\uG,
$$
and the internal operations $\tau^a$, with $a\in M(S_n)$,
$$
  \tau^a\negmedspace :
  M\uG\xrightarrow{\phantom{xxx}\tau_n\phantom{xxx}} M(S_n)\tensor M\uG
  \xrightarrow{\phantom{xxx}\langle-,a\rangle\tensor\id\phantom{xxx}}  M\uG,
$$
where
$$
  \tau=\sum_{n\geq 0}\tau_n\,t^n.
$$
\begin{Rem}[Variations]
  If each $M(S_n)$ is self-dual via $\langle -,-\rangle$ then the
  internal operations $\tau^a$ with $a$ in (a basis of) $M(S_n)$
  determine $\tau_n$. If there is no K\"unneth isomorphism then
  $\tau_n$ takes values in $M(S_n\times \uG)$. In this case,
  one can still define the internal operations, using the push-forward
  along the map $\map{pr_2}{S_n\times \uG}{\uG}$. One can also replace
  $\langle-,a\rangle$ by any $M(pt)$-module map $\map\phi{M(S_n)}N$ to obtain an
  operation $\tau^\phi$ with values in $N\tensor M(\uG)$. This
  variation is useful when the inner product is not defined.
\end{Rem}
Hoffman proves that these operations satisfy the $\tau$-ring axioms:
\begin{description}
\item[Cartan formula:] $\tau(0)=1_0\tensor1$ and 
  \begin{eqnarray*}
    \tau(x+y) & = & \tau(x)\times\tau(y),
  \end{eqnarray*}
  where $\times$ is the cross product tensored with multiplication in
  $M(\uG)$. In particular, the map
  \begin{eqnarray*}
    M(\uH)\oplus M(\uG)&\longrightarrow & S_M(\uH)\,\,
    \widehat\tensor\,\,
    S_M\tensor M(\uG)\\
    (x,y) & \longmapsto & P(x)\tensor \tau(y) 
  \end{eqnarray*}
  is bilinear from $+$ to $\times$, yielding a unique exponential
  extension
  \begin{eqnarray*}
    M(\uH)\tensor M(\uG)&\longrightarrow & S_M(\uH)\,\,\widehat\tensor
    \,\,S_M\tensor M(\uG),
  \end{eqnarray*}
  which we will denote $P\tensor\tau$.
\item[Low degrees:] $\tau_0(x) = 1_0\tensor 1$ and $\tau_1(x) =
  1_1\tensor x$.
\item[External products:]
  \begin{eqnarray*}
    m\negmedspace : \tau_j(x)\tensor  \tau_k(x) & \longmapsto &
    \res{S_{j+k}}{S_j\times S_k}\tau_{j+k}(x),
  \end{eqnarray*}
  where $m$ is $id_{M(S_j\times S_k)}$ tensored with multiplication in $M(\uG)$.
\item[Composition:]
  \begin{eqnarray*}
    (P\tensor\tau)_j\circ \tau_k(x) &\longmapsto &
    \res{S_{jk}}{S_j\lwr S_k} \tau_{jk}(x)
  \end{eqnarray*}
  under the map
  \begin{eqnarray*}
  \!\!\!\!\!\!\!\!\!\!
   (id\cdot\Pi^*)\tensor id:
    M(S_j\lwr S_k)\tensor M(S_j)\tensor M(\uG) &\longrightarrow &
        M(S_j\lwr S_k)\tensor M(\uG), 
  \end{eqnarray*}
  where $\longmap\Pi{S_j\lwr S_k}S_j$ is the projection.
\item[Multiplicativity:] $\tau_j(1) = 1_j\tensor 1$ and $\tau_j(xy)
  =\tau_j(x)\cdot\tau_j(y)$ (juxtaproduct). 
\end{description}
Still following Hoffman, we find that consequences of these axioms include
\begin{description}
\item[Universal maps:] For fixed $x\in M(\uG)$, we have a ring map
  \begin{eqnarray*}
    U_x\negmedspace :S_M & \longrightarrow & M(\uG) \\
    a   & \longmapsto     & \tau^a(x)
  \end{eqnarray*}
  sending $1\in M(S_1)$ to $x$. Here the product on the source is the
  cross product.
\item[Outer plethysm:] Composition of the internal operations is given
  by the formula
  \begin{eqnarray*}
    \tau^b\circ\tau^a & = & \tau^{b\vee a}
  \end{eqnarray*}
  with 
  \begin{eqnarray*}
    \vee:R(S_j)\tensor R(S_k) & \longrightarrow & R(S_{jk})\\
    b\vee a & = & \nabla_*\(P_j(a)\cdot \Pi^*(b)\)
  \end{eqnarray*}
  the {\em outer plethysm} \cite[Prop.4.5]{Hoffman:tau-rings}. 
\end{description}
\begin{Def}
  A {\em $\tau$-ring} with respect to the global power functor $M$ is
  a commutative and unitary $M(\pt)$-algebra $E$ together with an
  operation  
  \begin{eqnarray*}
    \tau\negmedspace : E & \longrightarrow & S_M \tensor E
  \end{eqnarray*}
  satisfying the $\tau$-ring axioms.
  If $M$ is clear from the context
  {\em `$\tau$-ring'} will always be meant with respect to $M$.
\end{Def}
The discussion above implies that the ring $(S_M,\times,1_0)$
with the operations
\begin{eqnarray*}
  \tau_j(b)   & = & (\nabla,\Pi)_*\, P_j(b)\quad\quad\text{and}\\
  \tau^a(b) & = & a\vee b 
\end{eqnarray*}
is the free $\tau$-ring on the generator $1_1$.
Here $\nabla$ is the inclusion of $S_j\lwr S_k$ in $S_{jk}$ and $\Pi$ is
the projection to $S_j$, where the degree of $b$ is $k$.

\medskip
On any $\tau$-ring $E$, we have the operations
\begin{eqnarray*}
  \on{sym}^n & = & \tau^{1_n} \quad\quad \text{ symmetric powers,} \\
  \lambda^n  & = & \tau^{f_n} \quad\quad\text{ $\lambda$-operations, and} \\
  \psi^n     & = & \tau^{c_n} \quad\quad\text{ Adams operations.}
\end{eqnarray*}
Here $1_n$, $f_n$ and $c_n$ are as in Section \ref{sec:elements}.
The $\lambda^n$ make $E$ a (not necessarily special)
$\lambda$-ring. In particular, the Adams operations are additive. 
We also have the operations
\begin{eqnarray*}
    \psi^H     & = & \tau^{\eta(X_H)},
\end{eqnarray*}
where $H$ is a finite group and $\eta(X_H)$ is as in Example
\ref{exa:chromatic_characters}. Write $P_H(x)$ for the restriction of
$P_n(x)$ along the inclusion
\begin{eqnarray*}
  H\times\uG &\sub & S_n\lwr \uG,
\end{eqnarray*}
and let $\map{\epsilon_H}{H\times\uG}{\uG}$ be the projection to the
second factor. Then, by Frobenius
reciprocity, 
\begin{eqnarray*}
  \psi^H(x) & = & (\epsilon_H)_!\(P_H(x)\).
\end{eqnarray*}
%

\subsection{Examples}\label{Powerops-examples-Sec}
\subsubsection{(Special) $\lambda$-rings}
The representation ring functor $\uG\mapsto R(\uG)$
is a global power functor with
$$P_n(V) = V^{\tensor n}$$ \cite{Atiyah:Power_operations}.
Here $sym^n$ and $\lambda^n$ are the familiar symmetric and exterior
powers, and $\psi^n$ are the Adams operations.
Hoffman shows that
$R$-theoretic $\tau$-rings are the same thing as special
$\lambda$-rings.\footnote{This term is somewhat old-fashioned. Nowadays, `special $\lambda$-rings' are often
  simply referred to as `$\lambda$-rings'. Since we are interested also in
  non-special examples, we stick with the old-fashioned terminology.}
The universal map
\begin{eqnarray*}
  U_{\id}: S_R & \longrightarrow & R(U) \quad\text{or}\\
  U_{X}: S_R & \longrightarrow & {\on{lim}\limits_m}\, \ZZ[x_1,\dots,x_m]^{S_m}
\end{eqnarray*}
$X=x_1+x_2+x_3+\dots$ is the Schur-Weyl map.\footnote{Note that there
  is a mismatch here with our conventions for the inner 
products: to get the correct map, one needs $\langle
\varrho,\vartheta\rangle = Hom_G(\varrho,\vartheta)$. Switching
conventions amounts to passing from $V_\lambda$ to $V_\lambda^*$.}  
\subsubsection{Super $\lambda$-rings, Sergeev-Yamaguchi duality}
\begin{Def}
  A $\ZZ/2\ZZ$-graded groupoid is a groupoid $\uG$ together with a map
  of groupoids
  \begin{eqnarray*}
    \uG & \longrightarrow & \ZZ/2\ZZ.
  \end{eqnarray*}
\end{Def}
For such $\uG$, we let $R(\uG)$ be the Grothendieck group of
super-representations of $\uG$ with the additional relation 
\begin{eqnarray*}
  [V] + [\Pi V] & = & 0,
\end{eqnarray*}
where $\Pi$ is the shift functor.\footnote{This differs from the
  convention in \cite{Jozefiak:Characters_of_projective} and
  \cite[12.18 ff.]{Kleshchev:Linear_and_projective} by a sign.} We  
recall from \cite{Ganter:Kapranov:Symmetric} that $R(\uG)$ itself may be viewed the degree
zero part of a
superring: let $\mC_1 = \CC[\xi_1]/\xi_1^2$ be the Clifford
superalgebra $|\xi_1| = 1$, and write $\tensor_s$ for the super tensorproduct. Then
\begin{eqnarray*}
  R(\uG)_1 & = & R(\CC[\uG] \tensor_s\mathcal C_1)
\end{eqnarray*}
The product map
\begin{eqnarray*}
  \circledast\negmedspace: R(\uG)_\bullet \tensor_s R(\uH)_\bullet &
  \xrightarrow{\phantom{XX}\cong\phantom{XX}}
  & R(\uG\times\uH)_\bullet
\end{eqnarray*}
is defined as follows (compare \cite[12.21]{Kleshchev:Linear_and_projective}): If 
$V$ and $W$ are irreducible and
$deg[V]\cdot deg[W] = 0$, then 
\begin{eqnarray*}
  V\circledast W  & = &  V\tensor_s W.
\end{eqnarray*}
If $deg(V)=deg(W) = 1$, then $V\circledast W$ is the inverse image of
$V\tensor_s W$ under the
periodicity isomorphism
\begin{eqnarray*} 
  R(\uG) & \xrightarrow{\phantom{XX}\cong\phantom{XX}} & R(\CC[\uG]\tensor_s \mC_2), 
\end{eqnarray*}
$\mC_2 = \mC_1^{\tensor_s2}$.
In other words, 
\begin{eqnarray*}
  V\circledast W & = & (V\tensor_s W)^+
\end{eqnarray*}
is the $+1$-eigenspace of $i\xi_1\xi_2$ inside $V\tensor_s W$.

The product $\circledast$ makes $R(-)_\bullet$ a globally defined
Green functor on the category of $\ZZ/2\ZZ$-graded groupoids. In
particular, each $R(\uG)_\bullet$ is a superring.

We have power operations on $R_\bullet$, defined as follows:
\begin{eqnarray*}
  P^+_k\negmedspace: R(\uG)_0 & \longrightarrow & R(S_k\lwr
  \uG)_{0}\\
  P^-_k\negmedspace: R(\uG)_1 & \longrightarrow & R^-(S_k\lwr
  \uG)_{\ol k}\\
  V &\longmapsto & V^{\circledast k},
\end{eqnarray*}
where $\ol k$ is the parity of $k$, and  $S_k$ acts by permuting the
factors. In the case of $P^-$ this 
makes $V^{\circledast k}$ a spin representation of $\widetilde S_k$ (in a
somewhat non-canonical manner, see \cite{Yamaguchi:Duality}).

$R(\uG)$ is the prototype of a {\em super $\lambda$-ring}: a super ring
$A_\bullet$ together with operations
\begin{eqnarray*}
  \tau^+\negmedspace : A_0 & \longrightarrow & \(A_\bullet\tensor_s S_{R_\bullet}^+\)_0\\
  \tau^-\negmedspace : A_1 & \longrightarrow & A_\bullet\tensor_s S_{R_\bullet}^-
\end{eqnarray*}
with $\tau^-_k$ taking values in degree $k$ and
such that $\tau = (\tau^+,\tau^-)$ satisfies the $\tau$-ring axioms.
\begin{Exa}[Sergeev-Yamaguchi duality]
  Let $A=R(\mathfrak q(n))_\bullet$ be the representation ring of the
  queer Lie superalgebra $\mathfrak q(n)$, and
  let $V$ be the defining representation of $\mathfrak q(n)$. The main
  result of \cite{Yamaguchi:Duality} may be summarized as
  \begin{eqnarray*}
    [V^{\circledast k}] & = & \sum_{\stackrel{\lambda\,\, \dashv\,\,
        k}{strict}} [V_\lambda]\circledast [W_\lambda],
  \end{eqnarray*}
  where $V_\lambda$ runs through the irreducible spin superrepresentations
  of $\widetilde S_n$ and the $W_\lambda$ are pairwise distinct
  irreducible superrepresentations of $\mathfrak q(n)$.
  It follows that the universal map
  \begin{eqnarray*}
    U_V^-\negmedspace : S_{R_\bullet}^- & \longrightarrow &
    R(\mathfrak q(n))_\bullet \\
  \end{eqnarray*}
  maps $a=[V_\lambda]$ to 
  \begin{eqnarray*}
    (\tau^-)^{[V_\lambda]}(V)& =& [W_\lambda]. 
  \end{eqnarray*}
\end{Exa}
\begin{Rem}
  We chose here to view $R(\uG)_\bullet$ as a $\ZZ/2\ZZ$-graded
  theory. This made it a natural thing to consider {\em `Yamaguchi
    power operations'} $V\mapsto V^{\circledast k}$ as
  above. Yamaguchi explains in \cite[Sec.3]{Yamaguchi:Duality} how to
  translate between his picture and that of Sergeev: in the language
  of \cite[Sec (1.7)]{Ganter:Kapranov:Symmetric}, the {\em `Sergeev power
    operations'} act on the $\mathbb
  N$-graded theory  
  \begin{eqnarray*}
    R(\uG)_m & \longrightarrow & R^-(S_k\times \uG)_{mk}\\
    V &\longmapsto & V^{\tensor_s k}
  \end{eqnarray*}
  and correspond to the Yamaguchi power operations under the 
  periodicity isomorphism.
\end{Rem}

\subsubsection{$n$-special $\lambda$-rings}
\label{sec:nlambda}
As above, we view 
$$
  \chi \in  \on{n-Class}(\uG,R)^{SL_n(\ZZ)}
$$ 
as ${SL_n(\ZZ)}$-invariant global function on $Prin_{\uG}(T^n)$. 
\begin{Def}
  The function $P_k(\chi)$ on
  $\on{\uG-Cov_k}(T^n)$ is defined by
  \begin{eqnarray*}
    (P_k\chi)\(P\xrightarrow{\pi}Y\xrightarrow{\phi} T^n\) & = &
    \prod_{Y=\coprod Y_i}\chi(P\at{Y_i}\xrightarrow{\pi_i}Y_i),
  \end{eqnarray*}
  where the product runs over all connected components of $Y$.
\end{Def}
To make sense of the factors on the right-hand side, 
pick, for each $i$, a basis of the torus $Y_i$. Since $\chi$ is
invariant under the $Aut(\ZZ^n)$-actions, its value at $\pi_i$ is
independent of this identification $Y_i\cong T^n$.
We will refer to the corresponding $\tau$-rings as {\em $n$-special
  $\lambda$-rings}. 
The internal power operations on an $n$-special
$\lambda$-ring are generated by the {\em `elementary'} power
operations $\tau^{\ul\sigma(\alpha)}$, where the $n$-tuple
$$%
  \xymatrix{
  \ul\sigma(\alpha):\ZZ^n\,\ar@{->>}[rr]^{\,\,\,\,\,\,\alpha} &&
  {\,\,A\,\,\,} \ar@{>->}[rr] && \,\,S_{|A|} 
  }
$$
is viewed as the linear form on $\on{n-Class}(S_{|A|},R)$ sending
$\chi$ to $\chi(\ul\sigma(\alpha))$.
These are calculated as follows: Let
\begin{eqnarray*}
  \varphi\negmedspace : ker(\alpha) & \hookrightarrow &\ZZ^n
\end{eqnarray*}
be the inclusion, and choose any oriented
basis of the lattice $ker(\alpha)$. For $\map\beta{\ZZ^n}\uG$, we
then have
\begin{eqnarray*}
  \(\tau^{\ul\sigma(\alpha)}(\chi)\)(\beta) & = & \chi\(\beta\circ\phi^t\).
\end{eqnarray*}
If $R$ contains $\mathbb Q$, so inner products exist, then we have
symmetric and exterior powers, and the discussion in
\cite[Sec.9]{Ganter:Orbifold_genera} goes through to prove that the Adams
operators of the theory are
\begin{eqnarray*}
  \psi^k & = & \sum_{|A|=k}\tau^{\ul\sigma(\alpha)}.
\end{eqnarray*}

\subsubsection{$\beta$-rings and subgroup characters}
The Burnside ring functor $A(\uG)$ is a global power functor via
\begin{eqnarray*}
  {P_k}\negmedspace :[X]&\longmapsto& [X^k].
\end{eqnarray*}
The K\"unneth maps for $A$ are not isomorphisms, but one can still
make sense of a ring with internal $\tau$-operations with respect to
$A$.
These internal operations are generated by the $\beta$-operations 
\begin{eqnarray*}
  \beta_H & = & \tau^{[S_k/H]},  
\end{eqnarray*}
where $H$ ranges over the conjugacy classes of (transitive) subgroups of $S_k$.
This point of view is taken in \cite{Rymer:Power_operations} and was
developed further in \cite{Ochoa:Outer_plethysm}, \cite{Morris:Wensley:Adams_operations},
\cite{Vallejo:The_free_beta-ring}, \cite{Guillot:Adams_operations}, 
\cite{Bouc:Roekaeus}, and \cite{Morava:Santhanam}. 
The resulting formalism is that of $\beta$-rings.
These can be rather complicated objects, Guillot shows that
examples include the stable homotopy groups of spheres and, more
generally, the stable cohomotopy ring of a space.

Let us now turn to the related Mackey functor $SCF$ of subgroup class
functions, where the subgroup characters of permutation
representations take their values. An idea going back to Knutson is to
look for a Schur-Weyl type map out of $\mathbf B$, using $SCF$. For
this, Knutson suggests a definition of Adams operations on $SCF$ and
conjectures that the subgroup character map $\eta^{SCF}$ is a map of
$\lambda$-rings. The above discussion suggests a definition of power
operations (and hence Adams-operations) on $SCF$, which is different
to the structure anticipated by Knutson. The transformation
$\eta^{SCF}$ preserves power operations. However, since the
surjections axiom is not satisfied, $\eta^{SCF}$ does not commute with
transfers along the maps $\map{\epsilon_G}{\uG}1$, so that care must
be taken when trying to use it to analyze internal operations. 

For simplicity, we let $\uG$ be a group $G$.
Let $H$ be a subgroup of $S_n\lwr G$. Then 
we have an $n$-fold covering space $Y$ of $BH$, together with a
principal $G$-bundle $\map\pi PY$.
The space $Y$ is connected if
and only if the image of the projection $H\to S_n$ is transitive. In
that case we have $Y=EH/K$ for some subgroup $K\sub H$ with index
$[H:K] = n$, and $P$ is classified by a map $K\to G$, well defined up
to inner automorphism of $G$. The image of this map defines an element $[K]_G$ of
$[\mathcal S(G)]$.
If $H$ is not transitive, then $H=\prod H_i$ is the product of
transitive subgroups $H_i\sub S_{n_i}$, and we get subgroups $K_i\sub
H_i$ of index $n_i$, each of which gives an element of $[\mathcal S(G)]$. 
\begin{Def}
  Define power operations 
  \begin{eqnarray*}
   P_n: SCF(G) & \longrightarrow & SCF(S_n\lwr G)\\
     \chi&\longmapsto& \(P_n(\chi):[H]  \longmapsto \prod_{i}\chi\([K_i]_G\)  \).
  \end{eqnarray*}
\end{Def}
In the spirit of \cite{Guillot:Adams_operations}, Knutson's
`Schur-Weyl' map should take values in the Burnside ring of the
infinite general linear group of the field with one element.
A possible candidate is the map
\begin{eqnarray*}
  U\negmedspace :\mathbf B & \longrightarrow & \lim_{m\in\mathbb
    N}A(S_m).     
\end{eqnarray*}
Here $[m]$ is the set
with $m$ elements and the defining $S_m$-action,
and $U_{[m]}$ is the map
\begin{eqnarray*}
  U_{[m]}\negmedspace :\mathbf B & \longrightarrow & A(S_m)\\
  \left[S_k/H\right]&\longmapsto &\tau^{\eta_H}([m]).
\end{eqnarray*}
Note that we have used $\tau^{\eta_H}$ instead of $\tau^{[S_k/H]}$, 
where 
\begin{eqnarray*}
  {\eta_H}\negmedspace :{A(S_k\times \uG)}&\longrightarrow &A(\uG)
\end{eqnarray*}
picks out the subset where
$S_k$ acts with orbit type $S_k/H$. 
To study the map $U$, one needs an understanding of the simultaneous
decomposition of 
\begin{eqnarray*}
  [m]^k&=&Map([k],[m])
\end{eqnarray*}
into $S_k$- and $S_m$-orbit types. 
Alternatively, one could modify this construction to take values in
$$
   \lim_{m\in\mathbb N}SCF(S_m). 
$$
\subsubsection{ $G$-modular functions}
Let $F$ be a global function on $\mM_G$.
Then the global function $P_n(f)$ on $\mM_{\Sn\lwr G}$ is defined as follows: 
Given an elliptic curve $E$ and a principal $\Sn\lwr G$-bundle
$\map\psi PE$, we set
\begin{eqnarray*}
    P_n(F)(\psi) & = & \prod_{Y=\coprod Y_i}F(\pi_i),
\end{eqnarray*}
where the sum is over the connected components of $Y=[\uG\ltimes P]$. 
Here we are taking advantage of the fact that the maps $\map{\phi_i}{Y_i}{E}$ are
isogenies and that this identification is canonical after choosing
basepoints of the $Y_i$, a choice which does not affect $F(\pi_i)$. 
The resulting internal power operations play a big role in
\cite{Ganter:Hecke_operators}. In particular, the Adams operations of the
theory are the $G$-equivariant Hecke operators.
\subsubsection{Rezk's elliptic $\lambda$-rings}
\label{sec:Rezk}
We will work with $q$ inverted.
In \cite{Ganter:Orbifold_K-Tate}, we defined the total power operation on
$K_{\Tate}$ as follows
\begin{eqnarray*}
    P^{Tate}_t (F)(q)  &=& \prod^\times_{k\geq 1}
    P_{t^k}^{Atiyah}\(
    \theta_k(F)\)(q),
\end{eqnarray*}
where $\theta_k(F)(q)$ is the power series over $R\(\mI(\ul G)[\xi^\frac1k]\)$
with character
\begin{eqnarray*}
  (\theta_kF)\(g,h g^\frac bk\,;q \) &=& F\(g,h;q^\frac1k\zeta_k^b\),
\end{eqnarray*}
and $P^{Atiyah}$ is defined by
\begin{eqnarray*}
  P^{Atiyah}_n(Vq^\frac mk) & = & V^{\tensor n}q^{\frac{mn}k}.
\end{eqnarray*}
Again, the Adams operations of the theory are the ($q$-expansions of)
equivariant Hecke operators. The total symmetric power of $[V]\in
R(\uG)\sub K_{Tate}(\uG)$ is Witten's class
\begin{eqnarray*}
  S_{t}^{Tate}(V) &= & \bigotimes_{k\geq 0}S_{t^k}^{Atiyah}(V)
\end{eqnarray*}
(see \cite{Ganter:Orbifold_K-Tate}).
\begin{Def}
  Let $r_n$ be the regular
  representation of the centralizer of ${\varsigma_n}$ 
  and define the
  sequences $\delta_n$ and $\gamma_n$ by letting
  \begin{eqnarray*}
    \delta_n & = & ([1],\varsigma_n)\\     
    \gamma_n & = & ([\varsigma_n],r_n).
  \end{eqnarray*}
  Note that $\delta_n\in K_{Tate}(S_n)$, while $\gamma_n$ does not
  satisfy the rotation condition. Then any $K_{Tate}$ theoretic tau-ring has the
  internal operations
  \begin{eqnarray*}
    \nu^n & = & \tau^{\delta_n}
  \end{eqnarray*}
  (denoted $\psi^n$ in \cite{Rezk:quasi-elliptic}) and the operations
  \begin{eqnarray*}
    \mu^n & = & \tau^{\gamma_n}
  \end{eqnarray*}  
  mapping from $A$ to $A[q^{\frac 1n}]$.
\end{Def}
These operations make our $K_{Tate}$-theoretic tau rings into {\em elliptic
  lambda rings} as defined by Rezk \cite{Rezk:quasi-elliptic}.
In particular, 
$K_{Tate}(\ZZ/N\ZZ)$ is an elliptic lambda ring. This is the ring
$T[N]$ considered in \cite{Rezk:quasi-elliptic}, with
\begin{eqnarray*}
  spec\(K_{Tate}(\ZZ/k\ZZ)\) & \cong & Tate(q)[k]
\end{eqnarray*}
the scheme of $k$-torsion points of the Tate curve.

\begin{Thm}
  The $q$-expansion map of Section \ref{sec:q-exp} is a map of
  global power functors.
\end{Thm}
\begin{Pf}{}
  Let $F\in Mod(\uG)$. 
  It suffices to calculate the $q$-expansion of $\(P_nF\)\at{\widetilde
    H_n}$ in terms of the $q$-expansion of $F$.
  Let $\phi$ be an isogeny onto
  $\CC/\langle\tau,1\rangle$. 
  If $\phi$ is as in Example \ref{exa:tau+b/d} then the discussion
  there implies that 
  \begin{eqnarray*}
    (P_d(F)) \(\phi,(g,h);\tau\)
    & = & F\(g,h;\frac{\tau+b}{d}\).
  \end{eqnarray*}
  This is the same formula as for $\theta_d$.
  The isogeny $\phi$ corresponds to the commuting pair
  $(\varsigma_d,\varsigma_d^b)$ in $S_d$, so that $(\phi,g,h)$ is
  identified with $(g,hg^{\frac bd})\in\mI(\uG)[\xi^\frac1d]$.  

  If $\phi$ is as in Example \ref{exa:atau}, then 
  \begin{eqnarray*}
    (P_a(F))\(\phi,(g,h);\tau\) & = & F(g,h;a\tau).
  \end{eqnarray*}
  This is the same as $P^{Atiyah}_a(F)_{(1,\varsigma_a)}$
  An arbitrary isogeny may be written as composite of two isogenies
  as above.
\end{Pf}

\begin{Def}[Looijenga's symmetric theta functions]
  Let $G$ be a complex reductive algebraic group with maximal torus
  $T$ and Weyl group $W_0$. Let 
  \begin{eqnarray*}
    \widehat T &=& Hom(T,\CC^\times)\\
    \check T &=& Hom(\CC^\times,T)
  \end{eqnarray*}
  be the character and cocharacter lattices of $T$, and let
  $I$ be the minimal positive definite $W_0$-invariant symmetric
  bilinear form on $\check T$ satisfying 
  $\(\forall v\in\check T\)$
  \begin{equation}
    \label{even-Eqn}
    I(v,v)\in 2\ZZ.
  \end{equation}
  Write $$\longmap{I^\sharp}{\check T}{\widehat T}$$ for the adjoint of
  $I$.
  Consider formal power series $f(z,q)$ with exponents of $z$ in
  $\widehat T$ and exponents of $q$ in $\ZZ$, and $
  z^\lambda z^\mu=z^{\lambda+\mu}$ and $q^{m}q^n = q^{m+n}$, and let
  \begin{eqnarray*}
    \Theta_I& = & \left\{f(z,q)\mid \(\forall v\in \check
      T\)\(f(zq^v,q) =
      q^{-\frac{I(v,v)}2}z^{-I^\sharp(v)}f(z,q)\)\right\}. 
  \end{eqnarray*}
  Looijenga's ring of symmetric theta functions is 
  \begin{eqnarray*}
    \Theta^{W_0} & = & \sum_{k\geq0}\Theta_{kI}^{W_0}
  \end{eqnarray*}
  ($W_0$-invariant elements).
\end{Def}
An example is the basic theta function
\begin{eqnarray*}
  \theta_I & = & \sum_{v\in\check T} q^{\frac{I(v,v)}2}z^{I^\sharp(v)}.
\end{eqnarray*}
Elements of $\Theta^{W_0}$ should be interpreted as the sections of
the Looijenga line bundle on $\check T\tensor Tate(q)$ and hence as
the coefficients of $G$-equivariant Tate $K$-theory. These have yet to
be defined, but see \cite{Rezk:Lectures_on_the_Tate_curve}.
The action of isogenies of the Tate curve on Looijenga theta functions
was calculated in \cite[Part
III]{Ando:POECLG}. These should be internal
$\tau$-operations with respect to $K_{Tate}$.

The ring $\Theta^{W_0}$ is the elliptic cohomology analogue of the
rings $R(U(n))$. The basic theta function $\theta_I$ plays the role of
the defining representation $[\CC^n]\in R(U(n))$. Hence we expect a
Schur-Weyl map
\begin{eqnarray*}
  U_{\theta_I}\negmedspace : S_{K_{Tate}} & \longrightarrow & \Theta^{W_0},
\end{eqnarray*}
which should be accessible via Ando's calculations.
It would be interesting to study this map in more detail.

\subsubsection{Generalized cohomology theories}
If the generalized cohomology theory $E$ is equipped 
with $\hinf$-structure, then $E(\on{Borel}(\uG))$ is a global power functor. 
This is true for many familiar cohomology
theories, such as ordinary cohomology, 
Lubin-Tate-Morava $E$-theory, 
topological modular forms, 
$K$-theory, and cobordism. 
The standard reference for $\hinf$-spectra is
\cite{Bruner:May:McClure:Steinberger}, for a discussion of power operations on the
Morava-Lubin-Tate theories and the corresponding $\lambda$-ring
operations, see \cite{Ando:Isogenies}, \cite{Baker:Hecke_algebras}, 
\cite{Rezk:Power_operations_for_Morava_E-theory}, \cite{Rezk:Rings_of_power_operations},
and \cite{Ganter:Orbifold_genera}.
\subsubsection{Character maps}
\begin{Prop}
  The character map
  \begin{eqnarray*}
    char\negmedspace : R & \Longrightarrow & 1Class(-,\ZZ[\mu_\infty])^{Aut(\ZZ)}
  \end{eqnarray*}
  preserves power operations.
\end{Prop}
\begin{Pf}{}
  Let $\varrho$ be a $\uG$-representation, and recall that
  $P_n(\varrho)$ is the $S_n\lwr\uG$-representation $\varrho^{\tensor
    n}$.
  We need to show
  \begin{eqnarray*}
   \chi_{\varrho^{\tensor n}}(\sigma;g_1,\dots,g_n)
   &=&\prod_{(i_1,\dots,i_k)}\chi_\varrho(g_{i_k}\cdots g_{i_1}),
  \end{eqnarray*}
  where the product is over all cycles of $\sigma$.
  Indeed, the action of $(\sigma;g_1,\dots,g_n)$ on
  $V^{\tensor n}$ can be written
  as the tensor product over the cycles of sigma
  of actions of $(\varsigma_k;g_{i_{1}},\dots,g_{i_k})$ on $V^{\tensor k}$,
  with
  \begin{eqnarray*}
    \varsigma_k = (1,\dots,k)\in S_k.
  \end{eqnarray*}
  Further
  \begin{eqnarray*}
    (\varsigma_k;g_{i_1},\dots, g_{i_k})&\sim & (\varsigma_k;1,\dots,1,g),
  \end{eqnarray*}
  where
  $g=g_{i_k}\cdots g_{i_1}$. Let $B=(e_1,\dots,e_d)$ be a basis of $V$. Then
  $(\varsigma_k;1,\dots,g)$ sends the basis element
  $e_{j_1}\tensor\dots\tensor e_{j_k}$ of $V^{\tensor k}$ to 
  $$
    \(g e_{j_k}\)\tensor e_{j_1}\tensor\dots\tensor e_{j_{k-1}}.
  $$
  This basis element can only make a non-trivial
  contribution to the trace on $V^{\tensor k}$ if we have ${j_1} = \dots = {j_k}$,
  and in this case the contribution to the trace is the $j_1^{th}$
  diagonal entry of $g$. We conclude that the trace of
  $(\varsigma_k;1,\dots,1,g)$ on $V^{\tensor k}$ equals the trace
  of $\on{Mat}_B(g)$ on $V$.
\end{Pf}{}

It would be nice to have a more conceptual proof of the proposition.
The Hopkins-Kuhn-Ravenel character maps
  \begin{eqnarray*}
    char\negmedspace : E_n(B\uG) & \longrightarrow &
    \on{n-Class_p}(\uG;L)^{Aut(\widehat{\ZZ}_p)^n }
  \end{eqnarray*}
form a natural transformation of Green
functors. Note that $Aut(\widehat Z^n_p)$ also acts on $L$, so this
is not a special case of the discussion in Section \ref{sec:nlambda}
Instead, the situation is closer to $Mod(\uG)$. The effect of $char$
on power operations is described using isogenies of formal groups. This
is the subject of \cite{Ando:Isogenies} and \cite{Ando:Hopkins:Strickland:sigma}.
\appendix
\section{Comparison to Webb's definition}\label{Webb-def-Sec}
%
%
%
%
%
%
%
In this section, we will compare our definition of global Mackey
functor with that of \cite{Webb:A_guide}. 
\begin{Prop}
  The axioms for global Mackey functors stated in Section
  \ref{global-Mackey-Sec} imply the Axioms for global Mackey functors
  formulated in \cite[8]{Webb:A_guide}, 
  where Webb's classes $\mathcal X$ and $\mathcal
  Y$ are chosen to contain all finite groups.\footnote{Webb attributes
  this formulation to Bouc \cite{Bouc:Foncteurs}.} 
\end{Prop}
\begin{Pf}{}
  The Mackey axiom (double coset formula) follows from Example
  \ref{double-coset-Exa}. One checks that a diagram of groups
  $$
    \xymatrix{
    G\ar@{->>}[r]^{\beta}  &
    H \\
    {\beta\inv(K)}  \ar@{>->}[u] \ar@{->>}[r]  &
    K\ar@{>->}[u] 
    }
  $$ 
  is equivalent to a fibred-product square in $\Gpdf$. Consider now a
  diagram of surjective maps of groups
  $$
    \xymatrix{
    G\ar@{->>}[r]^{\gamma\phantom{xxxxxxx}}  &
    {H/\ker{\alpha}\ker{\beta}} \\
    H  \ar@{->>}[u]^{\beta} \ar@{->>}[r]_{\alpha}  &
    K.\ar@{->>}[u]_{\delta} 
    }
  $$ 
  In general, this diagram is not a fibred-product square. However, the
  fibred product $\uF := G\times_{H/\ker\alpha\ker\beta}K$ is
  equivalent to a finite group (namely the pullback of the diagram in
  the category of finite groups), and under this equivalence, the map
  $H\to\uF$ becomes a 
  surjective map of groups. Now our fibred-product axiom together with
  the surjection axiom imply the desired push-pull property for this
  square. 
\end{Pf}

\begin{Prop}\label{Webb=>Prop}
  The axioms for global Mackey functor formulated in \cite{Webb:A_guide} imply ours.
\end{Prop}
\begin{Pf}{}
  Let $M$ be a global Mackey functor in the sense of \cite[8]{Webb:A_guide}.
  Set $M(\emptyset) = 0$.
  Recall that every groupoid is equivalent to a disjoint union
  of finite groups. Let $\uG$ be a connected groupoid, and let $x,y\in
  G_0$. 
  Let $g$ be an arrow from $x$ to $y$ in $G_1$. Then conjugation by
  $g$ defines an isomorphism from $\map{c^g}{\on{Stab}(x)}{\on{Stab}(y)}$. 
  For a different choice of arrow, $g'$, the two isomorphisms $c^g$
  and $c^{g'}$ will differ from each other by an inner automorphism of
  $\on{Stab(y)}$. Hence the (inverse) isomorphisms $(c^{g})_*$  and
  $(c^{g})^*$ are independent of the choice of $g$.  
  We set 
  $$
    M(\uG) := \lim_{x\in G_0}M(\on{Stab}(x)).
  $$ 
  Then for each $x\in G_0$, there is an isomorphism $\on{Stab}(x)\cong
  M(\uG)$. 
  This definition extends to maps between
  (connected) groupoids by restricting to stabilizers. By
  construction, 
  naturally isomorphic maps will yield the same result, and
  equivalences of groupoids will yield isomorphisms.  
  We extend $M$ to non-connected groupoids using our coproduct axiom
  (and the consequence for projections).
  Because of the coproduct axiom and the
  Note that our surjection axiom follows from Webb's axiom (4) in the
  special case that $G=K$.
  It suffices to check the
  fibred-product axiom for maps of finite
  groups $$H\stackrel\alpha\longrightarrow
  G\stackrel\beta\longleftarrow K.$$ 
  Each map of groups can be factored into a surjection
  followed by an injection, and fibred products compose as they
  should,
  so that we are reduced to three cases:
  \begin{description}
  \item[Case 1] Both maps are injective: this case follows from Webb's
    Axiom (5) (the Mackey axiom).
  \item[Case 2] One of the maps is injective and the other is
    surjective: this case follows from Webb's Axiom (3). 
  \item[Case 3] 
    If $\alpha$ and $\beta$ both are surjective, then the fibred
    product $H\times_G K$ is equivalent to a finite group, namely to
    $$
      L := \{(h,k)\mid \alpha(h) = \beta(k)\}.
    $$
    We have surjections $\map\gamma LK$ and $\map\delta LH$ and an
    isomorphism
    \begin{eqnarray*}
    L/(\ker\gamma)(\ker\delta) &\cong & 
    L/\{(h,k)\mid \alpha(h) = 1 = \beta(k)\}\\
    &\cong& G. 
    \end{eqnarray*}
    Using Webb's Axiom (4), we obtain the desired push-pull identity.
    \end{description}
\end{Pf}
\bibliographystyle{alpha}
\bibliography{lambdarings}
\end{document}